\documentclass[9pt]{article}
\usepackage{amsmath, amssymb, amsthm, graphicx, cite, comment}
\usepackage[lmargin=1in,rmargin=1in,tmargin=1in,bmargin=1in]{geometry}
\usepackage[arrow,curve,matrix,arc,2cell]{xy}
\UseAllTwocells
\DeclareFontFamily{U}{rsfs}{} \DeclareFontShape{U}{rsfs}{n}{it}{<->
rsfs10}{} \DeclareSymbolFont{mscr}{U}{rsfs}{n}{it}
\DeclareSymbolFontAlphabet{\scr}{mscr}
\def\mathscr{\scr}
\begin{document}
%%%%%%%%%%%%%%%%%%%%%%%%%%%%%%%%%%%%%%%%%%%%%%%%%%%%%%%%%%%%%%%%%%%%%%%%
%%%%%%%%%%%%%%%%%%%%%%%%%%     Macros      %%%%%%%%%%%%%%%%%%%%%%%%%%%%%
%%%%%%%%%%%%%%%%%%%%%%%%%%%%%%%%%%%%%%%%%%%%%%%%%%%%%%%%%%%%%%%%%%%%%%%%
\def\e#1\e{\begin{equation}#1\end{equation}}
\def\ea#1\ea{\begin{align}#1\end{align}}
\def\eq#1{{\rm(\ref{#1})}}
\theoremstyle{plain}% default
\newtheorem{thm}{Theorem}[section]
\newtheorem{prop}[thm]{Proposition}
\newtheorem{lem}[thm]{Lemma}
\newtheorem{cor}[thm]{Corollary}
\newtheorem{quest}[thm]{Question}
\theoremstyle{definition}
\newtheorem{dfn}[thm]{Definition}
\newtheorem{ex}[thm]{Example}
\newtheorem{rem}[thm]{Remark}
\numberwithin{equation}{section}
\def\dim{\mathop{\rm dim}\nolimits}
\def\supp{\mathop{\rm supp}\nolimits}
\def\cosupp{\mathop{\rm cosupp}\nolimits}
\def\rank{\mathop{\rm rank}}
\def\coh{\mathop{\rm coh}\nolimits}
\def\id{\mathop{\rm id}\nolimits}
\def\Hess{\mathop{\rm Hess}\nolimits}
\def\Crit{\mathop{\rm Crit}}
\def\GL{\mathop{\rm GL}}
\def\SO{\mathop{\rm SO}\nolimits}
\def\Ker{\mathop{\rm Ker}}
\def\Im{\mathop{\rm Im}}
\def\Aut{\mathop{\rm Aut}}
\def\End{\mathop{\rm End}}
\def\Hom{\mathop{\rm Hom}\nolimits}
\def\Perv{\mathop{\rm Perv}\nolimits}
\def\Perf{\mathop{\rm Perf}\nolimits}
\def\Perfis{\mathop{\text{\rm Perf-is}}\nolimits}
\def\Spec{\mathop{\rm Spec}\nolimits}
\def\Lbis{\mathop{\text{\rm Lb-is}}\nolimits}
\newcommand{\SExt}{\mathcal{E}xt}
\newcommand{\STor}{\mathcal{T}or}
\newcommand{\SHom}{\mathcal{H}om}

\def\sh{\sharp}
\def\red{{\rm red}}
\def\an{{\rm an}}
\def\alg{{\rm alg}}
\def\bs{\boldsymbol}
\def\ge{\geqslant}
\def\le{\leqslant\nobreak}
\def\bD{{\mathbin{\mathbb D}}}
\def\bH{{\mathbin{\mathbb H}}}
\def\bL{{\mathbin{\mathbb L}}}
\def\bP{{\mathbin{\mathbb P}}}
\def\A{{\mathbin{\mathcal A}}}
\def\PP{{\mathbin{\mathcal P}}}
\def\O{{\mathbin{\cal O}}}
\def\cA{{\mathbin{\cal A}}}
\def\cB{{\mathbin{\cal B}}}
\def\cC{{\mathbin{\cal C}}}
\def\cD{{\mathbin{\scr D}}}
\def\cE{{\mathbin{\cal E}}}
\def\cF{{\mathbin{\cal F}}}
\def\cG{{\mathbin{\cal G}}}
\def\cH{{\mathbin{\cal H}}}
\def\cL{{\mathbin{\cal L}}}
\def\cM{{\mathbin{\cal M}}}
\def\cO{{\mathbin{\cal O}}}
\def\cP{{\mathbin{\cal P}}}
\def\cS{{\mathbin{\cal S}}}
\def\cSz{{\mathbin{\cal S}\kern -0.1em}^{\kern .1em 0}}
\def\PV{{\mathbin{\cal{PV}}}}
\def\TS{{\mathbin{\cal{TS}}}}
\def\cQ{{\mathbin{\cal Q}}}
\def\cW{{\mathbin{\cal W}}}
\def\C{{\mathbin{\mathbb C}}}
\def\CP{{\mathbin{\mathbb{CP}}}}
\def\bA{{\mathbin{\mathbb A}}}
\def\K{{\mathbin{\mathbb K}}}
\def\Q{{\mathbin{\mathbb Q}}}
\def\R{{\mathbin{\mathbb R}}}
\def\Z{{\mathbin{\mathbb Z}}}
\def\al{\alpha}
\def\be{\beta}
\def\ga{\gamma}
\def\de{\delta}
\def\io{\iota}
\def\ep{\epsilon}
\def\la{\lambda}
\def\ka{\kappa}
\def\th{\theta}
\def\ze{\zeta}
\def\up{\upsilon}
\def\vp{\varphi}
\def\si{\sigma}
\def\om{\omega}
\def\De{\Delta}
\def\La{\Lambda}
\def\Si{\Sigma}
\def\Tau{{\rm T}}
\def\Th{\Theta}
\def\Om{\Omega}
\def\Ga{\Gamma}
\def\Up{\Upsilon}
\def\pd{\partial}
\def\db{{\bar\partial}}
\def\ts{\textstyle}
\def\st{\scriptstyle}
\def\sst{\scriptscriptstyle}
\def\w{\wedge}
\def\sm{\setminus}
\def\bu{\bullet}
\def\op{\oplus}
\def\ot{\otimes}
\def\otL{{\kern .1em\mathop{\otimes}\limits^{\sst L}\kern .1em}}
\def\boxtL{{\kern .2em\mathop{\boxtimes}\limits^{\sst L}\kern .2em}}
\def\boxtT{{\kern .1em\mathop{\boxtimes}\limits^{\sst T}\kern .1em}}
\def\ov{\overline}
\def\ul{\underline}
\def\bigop{\bigoplus}
\def\bigot{\bigotimes}
\def\iy{\infty}
\def\es{\emptyset}
\def\ra{\rightarrow}
\def\Ra{\Rightarrow}
\def\Longra{\Longrightarrow}
\def\ab{\allowbreak}
\def\longra{\longrightarrow}
\def\hookra{\hookrightarrow}
\def\dashra{\dashrightarrow}
\def\t{\times}
\def\ci{\circ}
\def\ti{\tilde}
\def\d{{\rm d}}
\def\ha{{\ts\frac{1}{2}}}
\def\md#1{\vert #1 \vert}
\def\ms#1{\vert #1 \vert^2}
\def\HM{\mathop{\rm HM}\nolimits}
\def\MHM{\mathop{\rm MHM}\nolimits}
\def\rat{\mathop{\bf rat}\nolimits}
\def\Rat{\mathop{\bf Rat}\nolimits}
\def\HV{{\mathbin{\cal{HV}}}}
%%%%%%%%%%%%%%%%%%%%%%%%%%%%%%%%%%%%%%%%%%%%%%%%%%%%%%%%%%%%%%%%%%%%%%%%
%%%%%%%%%%%%%%%%%%%%%% Text of paper %%%%%%%%%%%%%%%%%%%%%%%%%%%%%%%%%%%
%%%%%%%%%%%%%%%%%%%%%%%%%%%%%%%%%%%%%%%%%%%%%%%%%%%%%%%%%%%%%%%%%%%%%%%%
\title{{\bf{Categorification of Lagrangian intersections \\ on complex symplectic manifolds \\ using perverse sheaves of vanishing cycles}}}
\author{\smallskip\\  
Vittoria Bussi \\
\small{The Mathematical Institute,}\\
\small{Andrew Wiles Building
Radcliffe Observatory Quarter,}\\
\small{Woodstock Road, Oxford, OX1 3LB, U.K.} \\
\small{E-mail: \tt bussi@maths.ox.ac.uk}}
\date{\small{}}
\maketitle

\begin{abstract}

\smallskip

We study intersections of complex Lagrangian 
in complex symplectic manifolds, proving two main results.

\smallskip

First, we construct global canonical perverse sheaves on complex Lagrangian intersections
in complex symplectic manifolds for any pair of {\it oriented} Lagrangian submanifolds 
in the complex analytic topology. Our method uses classical results in complex symplectic geometry and some results from \cite{Joyc1}. 
The algebraic version of our result has already been obtained by the author et al. in \cite{BBDJS} using different methods, where we used, in particular, the recent new theory of algebraic d-critical loci introduced by Joyce in \cite{Joyc1}.
This resolves a long-standing question in the categorification of Lagrangian intersection numbers
and it may have important consequences in symplectic geometry and topological field theory.

\smallskip

Our second main result proves that (oriented) complex Lagrangian intersections
in complex symplectic manifolds for any pair of Lagrangian submanifolds 
in the complex analytic topology carry the structure of (oriented) analytic d-critical loci 
in the sense of \cite{Joyc1}.

\end{abstract}

\setcounter{tocdepth}{2}
\tableofcontents

\clearpage

\section*{Introduction} 
\markboth{Introduction}{Introduction}
\addcontentsline{toc}{section}{Introduction}
\label{s1}

Let $(S, \om)$ be a complex symplectic manifold, i.e., a complex manifold $S$
endowed with a closed non-degenerate holomorphic $2$-form $\om\in\Omega^2_S$.
Denote the complex dimension of $S$ by $2n.$
A complex submanifold $M \subset S$ is {\it Lagrangian} if 
the restriction of $\om$ to a $2$-form on $M$
vanishes and $\dim M = n.$
Let $X=L\cap M$ be the intersection as a complex analytic space. Then $X$
carries a canonical symmetric obstruction theory 
$\varphi: E^\bu \ra \bL_X$
in the sense of \cite{BeFa}, which can be represented by the complex $E^\bu \simeq[T^*S\vert_X\ra T^*L\vert_X\op T^*M\vert_X]$
with $T^*S\vert_X$ in degree $-1$ and $T^*L\vert_X\op T^*M\vert_X$
in degree zero. Hence
$\det(E^\bu)\cong K_L\vert_X\ot K_M\vert_X$. 
Inspired by \cite[\S 5.2]{KoSo1} in primis and then by \cite[\S 2.4]{BBDJS} and close to \cite[\S 5.2]{Joyc1}, we will say that 
if we are given square roots\/ $K_L^{1/2},K_M^{1/2}$ for $K_L,K_M$, 
then $X$ has {\it orientation data}. In this case we will also say that $L,M$ are {\it oriented} Lagrangians, see
Remark \ref{spin}. 

\medskip

We start from well known facts from complex symplectic geometry. 
It is well established 
that every complex symplectic manifold $S$ is locally isomorphic to the cotangent bundle $T^*N$ of a complex manifold $N.$ The fibres of the induced vector bundle structure on $S$ are Lagrangian submanifolds, so complex analytically locally defining on $S$ a foliation by Lagrangian submanifolds, i.e., a {\it polarization}. 
The data of a polarization for us will be used as a way to describe locally in the complex analytic topology the Lagrangian intersection $X$ as a critical locus $X\cong\Crit(f:U\ra \C),$ where $f$ is a holomorphic function on a complex manifold $U$. 
One moral of this approach is that every polarization defines a set of data for $X$
which we will call a {\it chart}, by analogy with critical charts defined by \cite[\S 2.1]{Joyc1},
and thus
the choice of a family of polarizations on a complex symplectic manifold provides a family of charts
which will be useful to defining some geometric structures on them and consequently get a global object on $X$ by {\it gluing}. This will become more clear later.  
In conclusion, on each chart defined by the choice of a polarization, there is naturally associated a perverse sheaf of vanishing cycles $\PV^\bu_{U,f}$ in the notation of \S\ref{s2}.

\medskip

Now, a natural problem to investigate is the following. Given analytic open 
${R}_i,{R}_j\subseteq X$ with isomorphisms 
${R}_i\cong\Crit(f_i)$, ${R}_j\cong\Crit(f_j)$ for holomorphic
$f_i:U_i\ra\C$ and $f_j:U_j\ra\C$, we have to understand whether the
perverse sheaves $\cP^\bu_{{R}_i}=\PV_{U_i,f_i}^\bu$ on 
${R}_i$ and $\cP^\bu_{{R}_j}=\PV_{U_j,f_j}^\bu$ on ${R}_j$
are isomorphic over ${R}_i\cap{R}_j$, and if so, whether
the isomorphism is canonical, for only then can we hope to glue the
$\cP^\bu_{{R}_i}$ for $i\in I$ to make $\cP^\bu_{L,M}$. Studying
these issues led to this paper.

\medskip

Our approach was inspired by a work of Behrend and Fantechi
\cite{BeFa}. They also investigated Lagrangian intersections in complex symplectic manifolds, but
their project is probably more ambitious, as they show the existence of deeply interesting structures carried 
by the intersection. Unfortunately, their construction has some crucial mistakes. Our project
started exactly with the aim to fix them and develop then an independent theory. 
In the meantime, the author worked with other collaborators on a large project \cite{BBDJS,BJM,BBJ}
involving Lagrangian intersections too, but our methods here want to be self contained and 
independent from that. In particular, the analogue of our Theorem below for algebraic symplectic manifolds and 
algebraic manifolds follows from \cite{BBJ,BBDJS,PTVV},
but the complex analytic case is not available in \cite{BBJ,PTVV}.

\medskip

In \S\ref{s3} we will state and prove the following result:

\medskip

\noindent\textbf{Theorem} {\it Let\/ $(S,\om)$ be a complex symplectic manifold and\/
$L,M$ oriented complex Lagrangian submanifolds in $S,$ and write $X=L\cap M,$
as a complex analytic subspace of\/ $S$. Then we may define
$P_{L,M}^\bu\in\Perv(X),$ uniquely up to canonical isomorphism, and
isomorphisms $\Si_{L,M}:P_{L,M}^\bu\ra \bD_X(P_{L,M}^\bu),$
$\Tau_{L,M}:P_{L,M}^\bu\ra P_{L,M}^\bu$,
respectively the Verdier duality and the monodromy
isomorphisms. These $P_{L,M}^\bu\in\Perv(X),\Si_{L,M},\Tau_{L,M}$ are characterized 
by the following property.

\smallskip

Given a choice of local Darboux coordinates $(x_1,\ldots,x_n,y_1,\ldots,y_n)$ 
in the sense of Definition \ref{basic} such that
$L$ is locally identified in coordinates with the graph $\Gamma_{\d f(x_1,\ldots,x_n)}$ of $\d f$ for $f$ a holomorphic function defined locally on an open $U\subset \C^n$, and $M$  is locally identified in coordinates with the graph  $\Gamma_{\d g(x_1,\ldots,x_n)}$ of $\d g$ for $g$ 
 a holomorphic function defined locally on $U$, and
the orientations $K_L^{1/2},K_M^{1/2}$ are the trivial square roots of
$K_L \cong \langle \d x_1\wedge \cdots\wedge\d x_n\rangle \cong K_M$, 
then $P_{L,M}^\bu\cong \PV^\bu_{U,g-f},$ where $\PV^\bu_{U,g-f}$
is the perverse sheaf of vanishing cycles of $g-f,$
and $\Sigma_{L,M}$ and $\Tau_{L,M}$ are respectively the Verdier duality $\sigma_{U,g-f}$ and the monodromy $\tau_{U,g-f}$
introduced in \S\ref{s2}.

\smallskip

The same applies for $\cD$-modules and mixed Hodge modules on\/~$X$.}

\medskip

Here is a sketch of the method of proof, given in detail in \S\ref{s3.1}--\ref{s3.3}.

\medskip

Given $(S,\om)$ a complex symplectic manifold we want to construct a global perverse sheaf $P_{L,M}^\bu\in\Perv(X),$ by gluing
together local data coming from choices of polarizations by isomorphisms. We consider an open cover $\{S_i\}_{i\in I}$ of $S$
and polarizations $\pi_i:S_i\ra E_i$, always assumed to be transverse to both the Lagrangians $L$ and $M.$ We use the following method:
\begin{itemize}
\setlength{\itemsep}{0pt}
\setlength{\parsep}{0pt}
\item[(i)] For each polarization $\pi_i: S_i\ra E_i$ transverse to both the
Lagrangian submanifolds $L$ and $M$,
we define a perverse sheaf of vanishing cycle $\PV^\bu_{f_i},$ naturally defined
on the chart induced by the choice of a polarization. 
and a principal $\Z_2$-bundle $Q_{f_i}$,
which roughly speaking parametrizes isomorphisms $K_L^{1/2} \cong K_M^{1/2}$ compatible with $\pi_i$.

\item[(ii)] For two such polarizations $E_i$ and $E_j$, transverse to each other, and to both the Lagrangians,
we have a way to define
two perverse sheaves of vanishing cycles, $\PV^\bu_{f_i}$ and
$\PV^\bu_{f_j},$ again with principal $\Z_2$-bundles,
each of them parametrizing choices of square roots of the canonical bundles of $L\cong\Gamma_{\d f_i}$
and $M\cong\Gamma_{\d f_j}.$ In this case we find an isomorphism $\Psi_{ij}$ on double overlap $S_i\cap S_j$
between $\PV^\bu_{f_i}\ot_{\Z_2} Q_{f_i}$ and $\PV^\bu_{f_j}\ot_{\Z_2} Q_{f_j}$.

\item[(iii)] For four such polarizations $E_i,$ $E_j,$ $E_k$ and $E_l$ with $E_i$ not necessarily transverse to $E_k,$ we obtain equality between $\Psi_{ij}\ci \Psi_{jk}$ and $\Psi_{il}\ci \Psi_{lk}$ on $S_i\cap S_j \cap S_k \cap S_l$.

As perverse sheaves form a stack in the sense of Theorem \ref{sm2thm3}, 
there exists $P_{L,M}^\bu$ on $X$, unique up to canonical
isomorphism, with $P^\bu_{L,M}\vert_{S_i}\cong\PV^\bu_{f_i}\ot_{\Z_2} Q_{f_i},$ for
all~$i\in I$.
\end{itemize}

Our perverse sheaf $P^\bu_{L,M}$ categorifies Lagrangian intersection numbers, in the sense that
the constructible function 
$$p\ra \sum_i (-1)^{i} \dim_\C \mathbb{H}^i_{\{ p \}}(X, P^\bu_{L,M}),$$
is equal to the well known Behrend function $\nu_X$ in \cite{Behr} by construction, using the expression of the Behrend function of a critical locus in terms of the Milnor fibre, as in \cite{Behr}, and so $$\chi(X,\nu_X)=\sum_i (-1)^{i}\dim_\C \mathbb{H}^i (X,P^\bu_{L,M}).$$

This resolves a long-standing question in the categorification of Lagrangian intersection number,
and it may have exciting far reaching consequences in symplectic geometry and topological field theory.

\medskip

In \cite{KR2}, Kapustin and Rozansky study boundary conditions and defects in a three-dimensional topological sigma-model with a complex symplectic target space, the Rozansky-Witten model.
They conjecture the existence of an interesting 2-category, the 2-category of boundary conditions. 
Their toy model for symplectic manifold is a cotangent bundle of some manifold. In this case, this category is related to the category of matrix factorizations \cite{Orlov}. 
Thus, we strongly believe that 
constructing a sheaf of $\Z_2$-periodic triangulated categories on Lagrangian intersection would yield an answer to their conjecture.
In the language of categorification, this would give a second categorification of the intersection numbers, the first being given by the hypercohomology of the perverse sheaf constructed in the present work. 
Also, this construction should be compatible with the Gerstenhaber and Batalin--Vilkovisky structures 
in the sense of \cite[Conj. 1.3.1]{BaGi}.

\medskip

Our second main result proved in \S\ref{s4} constitutes another bridge between our work and 
\cite{Joyc1,BBDJS,BBJ}.
Pantev et al.\ \cite{PTVV} show that {\it derived} intersections
$L\cap M$ of algebraic Lagrangians $L,M$ in an algebraic symplectic
manifold $(S,\om)$ have $-1$-shifted symplectic structures, so that
Theorem 6.6 in \cite{BBDJS} gives them the structure of algebraic
d-critical loci in the sense of \cite{Joyc1}. 
Our second main result shows a complex analytic version of this, 
which is not available from \cite{BBJ,PTVV},
that is,
 the {\it classical} intersection $L\cap M$ of complex 
Lagrangians $L,M$ in a complex symplectic manifold $(S,\om)$ has the structure of an (oriented) complex
analytic d-critical locus.

\medskip

\noindent\textbf{Theorem} {\it Suppose $(S,\om)$ is a complex
symplectic manifold, and\/ $L,M$ are (oriented) complex Lagrangian submanifolds
in $S$. Then the intersection $X=L\cap M,$ as a complex analytic
subspace of\/ $S,$ extends naturally to a (oriented) complex analytic
d-critical locus\/ $(X,s)$. The canonical bundle $K_{X,s}$ 
in the sense of \cite[\S 2.4]{Joyc1} is naturally isomorphic
to\/~$K_L\vert_{X^\red}\ot K_M\vert_{X^\red}$.}

\medskip

It would be interesting to prove an analogous 
version of this also for 
a class of `derived
Lagrangians' in~$(S,\om)$.
Some of the authors of \cite{BBDJS}
are working on defining a `Fukaya category' of (derived)
complex Lagrangians in a complex symplectic manifold, using
$\bH^*(P_{L,M}^\bu)$ as morphisms.

\medskip

\begin{center} \textbf{Outline of the paper} \end{center}  

\smallskip

The paper begins with a section of background material on perverse sheaves in the complex analytic topology.
Then we review basic notions in symplectic geometry. In \S\ref{s3}, we state and prove our first main result on the construction of a canonical global perverse sheaf on complex Lagrangian intersections. In \S\ref{s4} we prove our second main result on the d-critical locus structure carried by Lagrangian intersections. Finally, the last section
sketches some implications of the theory and proposes new ideas for further research.

\medskip

\begin{center} \textbf{Notations and conventions} \end{center}  

\smallskip

Throughout we will work in the complex analytic topology over 
$\C$. We will denote by $(S,\om)$ a complex symplectic manifold endowed with a symplectic form $\om,$
and its Lagrangian submanifolds will be always assumed to be nonsingular. 
Note that all complex analytic
spaces in this paper are locally of finite
type, which is necessary for the existence of embeddings
$i:X\hookra U$ for $U$ a complex manifold.
Fix a well-behaved commutative base ring $A$
(where `well-behaved' means that we need assumptions on $A$ such as
$A$ is regular noetherian, of finite global dimension or finite
Krull dimension, a principal ideal domain, or a Dedekind domain, at
various points in the theory), to study sheaves of $A$-modules. For
some results $A$ must be a field. Usually we take $A=\Z,\Q$ or~$\C$.

\medskip

\begin{center} \textbf{Acknowledgements} \end{center}  

\smallskip

I would like to thank my supervisor Dominic Joyce for his continuous support, for many enlightening suggestions and valuable discussions. I would like to thank also Kai Behrend and Barbara Fantechi for interesting discussions,
and Delphine Dupont, with whom I started to work on this project during her stay in Oxford.
This research is part of my D.Phil. project funded by an EPSRC Studentship.

\section{Background material}
\label{s2}

In this introductory section we first recall general definitions and conventions
about perverse sheaves on complex analytic spaces, and some results very well known in literature,
which will be used in the sequel. This first part has a substantial overlap with \cite{BBDJS}.
Secondly, we recall some definitions and results from \cite{Joyc1}, crucially used to prove one of the main result of \cite{BBDJS} about behavior of perverse sheaves of vanishing cycles under stabilization, stated in \S\ref{s3}.
Finally, we establish the basic notation for symplectic manifolds, their Lagrangian submanifolds and
polarizations, and we recall results from complex symplectic geometry, used in \S\ref{s3}.

\subsection{Perverse sheaves on complex analytic spaces}
\label{s2.1}

We discuss perverse sheaves on complex analytic spaces, as in
Dimca \cite{Dimc}.

\smallskip

For the whole section, let $A$ be a well-behaved commutative base ring (usually we take $A=\Z,\Q$ or~$\C$).
Usually $X$ will be a complex analytic space, always assumed locally of
finite type. We start by discussing constructible complexes, following
Dimca~\cite[\S 2--\S 4]{Dimc}.

\begin{dfn} 
A sheaf
$\cS$ of $A$-modules on $X$ is called {\it constructible\/} if all
the stalks $\cS_x$ for $x\in X$ are finite type $A$-modules, and
there is a locally finite stratification $X=\coprod_{j\in
J}X_j$ of $X$, where $X_j\subseteq X$ for $j\in J$ are
complex analytic subspaces of $X$, such that $\cS\vert_{X_j}$
is an $A$-local system for all~$j\in J$.

\smallskip

Write $D(X)$ for the derived category of complexes $\cC^\bu$ of
sheaves of $A$-modules on $X$. Write $D^b_c(X)$ for the full
subcategory of bounded complexes $\cC^\bu$ in $D(X)$ whose
cohomology sheaves $\cH^m(\cC^\bu)$ are constructible for all
$m\in\Z$. Then $D(X),D^b_c(X)$ are triangulated categories. An
example of a constructible complex on $X$ is the {\it constant
sheaf\/} $A_X$ on $X$ with fibre $A$ at each point.

\smallskip

Grothendieck's ``six operations on sheaves'' $f^*,f^!,Rf_*,Rf_!,
{\cal RH}om,\smash{\otL}$ act on $D(X)$ preserving the subcategory
$D^b_c(X)$ in case $f$ is proper. 

\smallskip

For $\cB^\bu,\cC^\bu$ in $D^b_c(X)$, we may form their {\it derived
Hom\/} ${\cal RH}om(\cB^\bu,\cC^\bu)$ \cite[\S 2.1]{Dimc}, and {\it
left derived tensor product\/} $\cB^\bu\otL\cC^\bu$ in $D^b_c(X)$,
\cite[\S 2.2]{Dimc}. Given $\cB^\bu\in D^b_c(X)$ and $\cC^\bu\in
D^b_c(Y)$, we define $\cB^\bu\smash{\boxtL}\cC^\bu=
\pi_X^*(\cB^\bu)\otL\pi_Y^*(\cC^\bu)$ in $D^b_c(X\t Y)$, where
$\pi_X:X\t Y\ra X$, $\pi_Y:X\t Y\ra Y$ are the projections.

\smallskip

If $X$ is a complex analytic space, there is a functor $\bD_X:D^b_c(X)\ra
D^b_c(X)^{\rm op}$ with $\bD_X\ci\bD_X\cong\id:D^b_c(X)\ra
D^b_c(X)$, called {\it Verdier duality}. It reverses shifts, that
is, $\bD_X\bigl(\cC^\bu[k]\bigr)=\bigl(\bD_X(\cC^\bu)\bigr)[-k]$ for
$\cC^\bu$ in $D^b_c(X)$ and~$k\in\Z$.

\smallskip

We will use the following property: if\/ $f:X\ra Y$ is a morphism then 
\e
Rf_!\cong \bD_Y\ci
Rf_*\ci\bD_X \quad \textrm{and}\quad f^!\cong \bD_X\ci f^*\ci\bD_Y.
\label{sm2thm1}
\e

If $X$ is a complex analytic space, and $\cC^\bu\in D^b_c(X)$, the {\it
hypercohomology\/} $\bH^*(\cC^\bu)$ and {\it compactly-supported
hypercohomology\/} $\bH^*_{\rm cs}(\cC^\bu)$, both graded
$A$-modules, are
\e
\bH^k(\cC^\bu)=H^k(R\pi_*(\cC^\bu))\quad\text{and}\quad \bH^k_{\rm
cs}(\cC^\bu)=H^k(R\pi_!(\cC^\bu))\quad\text{for $k\in\Z$,}
\label{sm2eq1}
\e
where $\pi:X\ra *$ is projection to a point. If $X$ is proper then
$\bH^*(\cC^\bu)\cong\bH^*_{\rm cs}(\cC^\bu)$. They are related to usual cohomology by
$\bH^k(A_X)\cong H^k(X;A)$ and $\bH^k_{\rm cs}(A_X)\cong H^k_{\rm
cs}(X;A)$. If $A$ is a field then under Verdier duality 
$\bH^k(\cC^\bu)\cong \bH^{-k}_{\rm cs}(\bD_X(\cC^\bu))^*$.
\label{sm2def1}
\end{dfn}

Next we review perverse sheaves, following Dimca~\cite[\S 5]{Dimc}.

\begin{dfn} Let $X$ be a complex analytic space, and for each $x\in X$,
let $i_x:*\ra X$ map $i_x:*\mapsto x$. If $\cC^\bu\in D^b_c(X)$,
then the {\it support\/} $\supp^m\cC^\bu$ and {\it cosupport\/}
$\cosupp^m\cC^\bu$ of $\cH^m(\cC^\bu)$ for $m\in\Z$ are
\begin{align*}
\supp^m\cC^\bu=\overline{\bigl\{x\in X:\cH^m(i_x^*(\cC^\bu))
\ne 0\bigr\}} \quad\textrm{and}\quad
\cosupp^m\cC^\bu=\overline{\bigl\{x\in X:\cH^m(i_x^!(\cC^\bu))
\ne 0\bigr\}},
\end{align*}
where $\overline{\{\cdots\}}$ means the closure in $X^\an$. If $A$
is a field then $\cosupp^m\cC^\bu=\supp^{-m}\bD_X(\cC^\bu)$. We call
$\cC^\bu$ a {\it perverse sheaf} if
$$\dim_{\sst\C}\supp^{-m}\cC^\bu\!\le\!m \quad\text{and}\quad \dim_{\sst\C}\cosupp^m\cC^\bu\!\le\! m$$  
for all $m\in\Z$, where by
convention $\dim_{\sst\C}\es=-\iy$. Write $\Perv(X)$ for the full
subcategory of perverse sheaves in $D^b_c(X)$. Then $\Perv(X)$ is an
abelian category, the heart of a t-structure on~$D^b_c(X)$.
\label{sm2def2}
\end{dfn}

Perverse sheaves have the following properties:

\begin{thm}{\bf(a)} If\/ $A$ is a field then $\Perv(X)$ is
noetherian and artinian.
\smallskip

\noindent{\bf(b)} If\/ $A$ is a field then $\bD_X:D^b_c(X)\ra
D^b_c(X)$ maps $\Perv(X)\ra\Perv(X)$.
\smallskip

\noindent{\bf(c)} If\/ $i:X\hookra Y$ is inclusion of a closed\/
complex analytic subspaces, then $Ri_*,Ri_!$ (which are naturally isomorphic)
map $\Perv(X)\ra\Perv(Y)$.

\smallskip

Write $\Perv(Y)_X$ for the full subcategory of objects in $\Perv(Y)$
supported on $X$. Then $Ri_*\cong Ri_!$ are equivalences of
categories $\Perv(X)\,{\buildrel\sim\over\longra}\, \Perv(Y)_X$. The
restrictions $i^*\vert_{\Perv(Y)_X},i^!\vert_{\Perv(Y)_X}$ map\/
$\Perv(Y)_X\ra \Perv(X),$ are naturally isomorphic, and are
quasi-inverses for\/~$Ri_*,Ri_!:\Perv(X)\ra\Perv(Y)_X$.
\smallskip

\noindent{\bf(d)} If\/ $f:X\ra Y$ is \'etale then $f^*$ and\/ $f^!$
(which are naturally isomorphic) map $\Perv(Y)\ra\Perv(X)$. More
generally, if\/ $f:X\ra Y$ is smooth of relative dimension\/ $d,$
then\/ $f^*[d]\cong f^![-d]$ map $\Perv(Y)\ra\Perv(X)$.
\smallskip

\noindent{\bf(e)} $\smash{\boxtL}:D^b_c(X)\!\t\! D^b_c(Y)\!\ra\!
D^b_c(X\!\t\! Y)$ maps\/~$\Perv(X)\!\t\!\Perv(Y)\!\ra\!\Perv(X\!\t\!
Y)$.

\smallskip

\noindent{\bf(f)} Let\/ $U$ be a complex manifold. Then
$A_U[\dim U]$ is perverse, where $A_U$ is the constant sheaf on $U$
with fibre $A,$ and\/ $[\dim U]$ means shift by $\dim U$ in the
triangulated category $D^b_c(X)$. Moreover, there is a canonical
isomorphism\/~$\bD_U\bigl(A_U[\dim U]\bigr)\cong A_U[\dim U]$.
\label{sm2thm2}
\end{thm}

The next theorem is proved in 
\cite[Th.~10.2.9]{KaSc1}, see also~\cite[Prop.~8.1.26]{HTT}. The analogue for $D^b_c(X)$ or $D(X)$
rather than $\Perv(X)$ is false. One moral is that perverse sheaves
behave like sheaves, rather than like complexes.

\begin{thm} Perverse sheaves on a complex analytic space $X$
form a \begin{bfseries}stack\end{bfseries} on $X$ in the complex analytic topology.
Explicitly, this means the following. Let $\{U_i\}_{i\in
I}$ be an analytic open cover for $X,$ and write
$U_{ij}=U_i \cap U_j$ for $i,j\in I$.
Similarly, write $U_{ijk}=U_i \cap U_j \cap U_k$ for $i,j,k\in
I$. With this notation:
\smallskip

\noindent{\bf(i)} Suppose $\cP^\bu,\cQ^\bu\in\Perv(X),$ and we are
given $\al_i : \cP^\bu\vert_{U_i} \ra \cQ^\bu\vert_{U_i}$ in $\Perv(U_i)$ for
all\/ $i\in I$ such that for all\/ $i,j\in I$ we have
$$\al_i \vert_{U_{ij}} = \al_j \vert_{U_{ij}},$$ then exists unique $\al:\cP^\bu\ra\cQ^\bu$ in $\Perv(X)$ with $\al\vert_{U_i}=\al_i$ for all\/~$i\in I$.

\smallskip

\noindent{\bf(ii)} Suppose we are given objects\/ $\cP^\bu_i\in
\Perv(U_i)$ for all\/ $i\in I$ and isomorphisms
$\al_{ij}:\cP^\bu_i \vert_{U_{ij}}\ra \cP^\bu_j\vert_{U_{ij}}$ in
$\Perv(U_{ij})$ for all\/ $i,j\in I$ with\/ $\al_{ii}=\id$ and\/
\begin{equation*}
\al_{jk}\vert_{U_{ijk}}\ci \al_{ij}\vert_{U_{ijk}}=
\al_{ik}\vert_{U_{ijk}}:\cP_i\vert_{U_{ijk}}\longra
\cP_k\vert_{U_{ijk}}
\end{equation*}
in $\Perv(U_{ijk})$ for all\/ $i,j,k\in I$. Then there exists\/
$\cP^\bu$ in $\Perv(X),$ unique up to canonical isomorphism, with
isomorphisms $\be_i: \cP^\bu \vert_{U_{i}} \ra\cP^\bu_i$ for each\/ $i\in I,$
satisfying $\al_{ij}\ci\be_i\vert_{U_{ij}}=\be_j\vert_{U_{ij}}:
\cP^\bu\vert_{U_{ij}}\ra\cP^\bu_j\vert_{U_{ij}}$ for all\/~$i,j\in I$.

\label{sm2thm3}
\end{thm}

If $P\ra X$ is a principal $\Z_2$-bundle on a complex manifold $X$, and
$\cQ^\bu\in\Perv(X)$, we will define a perverse sheaf
$\cQ^\bu\ot_{\Z_2}P$.

\begin{dfn} A {\it principal\/
$\Z_2$-bundle\/} $P\ra X$ on a complex analytic space $X$, is a proper, \'etale, surjective, complex analytic
morphism of complex analytic spaces $\pi:P\ra X$ together with a free
involution $\si:P\ra P$, such that the orbits of $\Z_2=\{1,\si\}$
are the fibres of $\pi$. 

\smallskip

Let $P\ra X$ be a principal $\Z_2$-bundle. Write $\cL_P\in
D^b_c(X)$ for the rank one $A$-local system on $X$ induced from $P$
by the nontrivial representation of $\Z_2\cong\{\pm 1\}$ on $A$.
It is characterized by $\pi_*(A_P)\cong A_X\op\cL_P$. For each
$\cQ^\bu\in D^b_c(X)$, write $\cQ^\bu\ot_{\Z_2}P\in D^b_c(X)$ for
$\cQ^\bu\otL\cL_P$, and call it $\cQ^\bu$ {\it twisted by\/} $P$. If
$\cQ^\bu$ is perverse then $\cQ^\bu\ot_{\Z_2}P$ is perverse. 
Perverse sheaves and complexes twisted by principal $\Z_2$-bundles
have the obvious functorial behavior.
 \label{sm2def3}
\end{dfn}

\smallskip

We explain nearby cycles and vanishing cycles, as in Dimca \cite[\S
4.2]{Dimc}. 

\begin{dfn} Let $X$ be a complex analytic space, and $f:X\ra\C$ a holomorphic
function. Define $X_0=f^{-1}(0)$, as a complex analytic subspace of
$X$, and $X_*=X\sm X_0$. Consider the commutative diagram of complex
analytic spaces:
\e
\begin{gathered}
\xymatrix@R=15pt@C=30pt{ X_0 \ar[r]_i \ar[d]^(0.45)f & X
\ar[d]^(0.45)f & X_* \ar[l]^j \ar[d]^(0.45)f &
{\widetilde{X_*}\phantom{.}} \ar[l]^p \ar@/_.7pc/[ll]_\pi
\ar[d]^(0.45){\ti f} \\ \{0\} \ar[r] & \C & \C^* \ar[l] &
{\widetilde{\C^*}.} \ar[l]_\rho}
\end{gathered}
\label{sm2eq2}
\e
Here $i:X_0\hookra
X$, $j:X_*\hookra X$ are the inclusions,
$\rho:\widetilde{\C^*}\ra\C^*$ is the universal cover of
$\C^*=\C\sm\{0\}$, and $\widetilde{X_*}=X_*\t_{f,\C^*,\rho}
\widetilde{\C^*}$ the corresponding cover of $X_*$, with
covering map $p:\widetilde{X_*}\ra X_*$, and $\pi=j\ci p$.
We define the {\it nearby cycle
functor\/} $\psi_f:D^b_c(X)\ra D^b_c(X_0)$ to be $\psi_f=i^*\ci
R\pi_*\ci \pi^*$.

\smallskip

There is a natural transformation $\Xi:i^*\Ra\psi_f $ between the
functors $i^*,\psi_f:D^b_c(X)\ra D^b_c(X_0)$. The {\it vanishing
cycle functor\/} $\phi_f:D^b_c(X)\ra D^b_c(X_0)$ is a functor such
that for every $\cC^\bu$ in $D^b_c(X)$ we have a distinguished
triangle
\begin{equation*}
\smash{\xymatrix@C=40pt{i^*(\cC^\bu) \ar[r]^{\Xi(\cC^\bu)} & \psi_f
(\cC^\bu) \ar[r] & \phi_f (\cC^\bu) \ar[r]^{[+1]} & i^*(\cC^\bu)}}
\end{equation*}
in $D^b_c(X_0)$. Following Dimca \cite[p.~108]{Dimc}, we write
$\psi_f^p,\phi_f^p$ for the shifted functors~$\psi_f[-1],\ab
\phi_f[-1]:D^b_c(X)\ra D^b_c(X_0)$.

\smallskip

The generator of $\Z=\pi_1(\C^*)$ on $\widetilde{\C^*}$ induces a
deck transformation $\de_{\smash{\C^*}}:\widetilde{\C^*}\ra
\widetilde{\C^*}$ which lifts to a deck transformation
$\de_{\smash{X^*}}:\widetilde{X^*}\ra\widetilde{X^*}$ with
$p\ci\de_{\smash{X^*}}=p$ and $\ti f\ci\de_{\smash{X^*}}=
\de_{\smash{\C^*}}\ci\ti f$. As in \cite[p.~103, p.~105]{Dimc}, we
can use $\de_{\smash{X^*}}$ to define natural transformations
$M_{X,f}:\psi^p_f\Ra\psi^p_f$ and $M_{X,f}:\phi^p_f\Ra\phi^p_f$,
called {\it monodromy}.

\smallskip

By Dimca
\cite[Th.~5.2.21]{Dimc}, if\/ $X$ is a complex analytic space and\/ $f:X\ra\C$ is
holomorphic, then\/ $\psi_f^p,\phi_f^p:D^b_c(X)\ra D^b_c(X_0)$ both
map\/~$\Perv(X)\ra\Perv(X_0)$.
\label{sm2def4}
\end{dfn}

We will use the following property, proved by
Massey \cite{Mass3}.
If\/ $X$ is a complex manifold and\/ $f:X\ra\C$ is
regular, then there are natural isomorphisms 
\e
\psi_f^p\ci\bD_X\cong
\bD_{X_0}\ci\psi_f^p \quad \textrm{and}\quad \phi_f^p\ci\bD_X\cong
\bD_{X_0}\ci\phi_f^p.
\label{sm2thm4}
\e

We can now define perverse sheaf
of vanishing cycles $\PV^\bu_{U,f}$ for a holomorphic
function~$f:U\ra\C$.

\begin{dfn} Let $U$ be a complex analytic space, and $f:U\ra\C$ a holomorphic
function. Write $X=\Crit(f)$, as a closed complex analytic subspace of
$U$. Then as a map of topological spaces, $f\vert_X:X\ra\C$ is
locally constant, with finite image $f(X)$, so we have a
decomposition $X=\coprod_{c\in f(X)}X_c$, for $X_c\subseteq X$ the
open and closed complex analytic subspace with $f(x)=c$ for each
$x\in X_c$.

\smallskip

For each $c\in\C$, write $U_c=f^{-1}(c)\subseteq U$. Then we have a vanishing cycle functor
$\phi_{f-c}^p:\Perv(U)\ra\Perv(U_c)$. So we may form
$\phi_{f-c}^p(A_U[\dim U])$ in $\Perv(U_c)$, since $A_U[\dim
U]\in\Perv(U)$ by Theorem \ref{sm2thm2}(f). One can show
$\phi_{f-c}^p(A_U[\dim U])$ is supported on the closed subset
$X_c=\Crit(f)\cap U_c$ in $U_c$, where $X_c=\es$ unless $c\in f(X)$.
That is, $\phi_{f-c}^p(A_U[\dim U])$ lies in~$\Perv(U_c)_{X_c}$.
But Theorem \ref{sm2thm2}(c) says $\Perv(U_c)_{X_c}$ and
$\Perv(X_c)$ are equivalent categories, so we may regard
$\phi_{f-c}^p(A_U[\dim U])$ as a perverse sheaf on $X_c$. That is,
we can consider $\phi_{f-c}^p(A_U[\dim U])\vert_{X_c}=i_{X_c,U_c}^*
\bigl(\phi_{f-c}^p(A_U[\dim U])\bigr)$ in $\Perv(X_c)$, where
$i_{X_c,U_c}:X_c\ra U_c$ is the inclusion morphism.
Also, $\Perv(X)=\bigop_{c\in f(X)}\Perv(X_c)$. 

\smallskip

Define the
{\it perverse sheaf of vanishing cycles\/ $\PV_{U,f}^\bu$ of\/}
$U,f$ in $\Perv(X)$ to be $$\PV_{U,f}^\bu=\bigop_{c\in
f(X)}\phi_{f-c}^p(A_U[\dim U])\vert_{X_c}.$$ That is, $\PV_{U,f}^\bu$
is the unique perverse sheaf on $X=\Crit(f)$ with
$\PV_{U,f}^\bu\vert_{X_c}=\phi_{f-c}^p(A_U[\dim U])\vert_{X_c}$ for
all~$c\in f(X)$.

\smallskip

Under Verdier duality, we have $A_U[\dim U]\cong\bD_U(A_U[\dim U])$
by Theorem \ref{sm2thm2}(f), so $\phi_{f-c}^p(A_U[\dim U])
\cong\bD_{U_c}\bigl(\phi_{f-c}^p(A_U[\dim U])\bigr)$ by 
\eq{sm2thm4}. Applying $i_{X_c,U_c}^*$ and using $\bD_{X_c}\ci
i_{X_c,U_c}^*\cong i_{X_c,U_c}^!\ci\bD_{U_c}$ by 
\eq{sm2thm1} and $i_{X_c,U_c}^!\cong i_{X_c,U_c}^*$ on
$\Perv(U_c)_{X_c}$ by Theorem \ref{sm2thm2}(c) also gives
\begin{equation*}
\phi_{f-c}^p(A_U[\dim U])\vert_{X_c}\cong \bD_{X_c}\bigl(
\phi_{f-c}^p(A_U[\dim U])\vert_{X_c}\bigr).
\end{equation*}
Summing over all $c\in f(X)$ yields a canonical isomorphism
\e
\si_{U,f}:\PV_{U,f}^\bu\,{\buildrel\cong\over\longra}\,\bD_X(\PV_{U,f}^\bu).
\label{sm2eq6}
\e

For $c\in f(X)$, we have a monodromy operator
$M_{U,f-c}:\phi_{f-c}^p(A_U[\dim U])\ab\ra \phi_{f-c}^p(A_U[\dim
U])$, which restricts to $\phi_{f-c}^p(A_U[\dim U])\vert_{X_c}$.
Define the {\it twisted monodromy operator\/}
$\tau_{U,f}:\PV_{U,f}^\bu\ra\PV_{U,f}^\bu$ by
\e
\begin{split}
\tau_{U,f}\vert_{X_c}&=(-1)^{\dim U}M_{U,f-c}\vert_{X_c}:\phi_{f-c}^p(A_U[\dim U])\vert_{X_c}\longra \phi_{f-c}^p(A_U[\dim
U])\vert_{X_c},
\end{split}
\label{sm2eq7}
\e
for each $c\in f(X)$. Here `twisted' refers to the sign $(-1)^{\dim
U}$ in \eq{sm2eq7}. We include this sign change as it makes
monodromy act naturally under transformations which change dimension
--- without it, equation \eq{sm5eq15} below would only commute up to
a sign $(-1)^{\dim V-\dim U}$, not commute --- and it normalizes the
monodromy of any nondegenerate quadratic form to be the identity. 
The sign $(-1)^{\dim U}$ also corresponds to the
twist `$(\ha\dim U)$' in the definition of the mixed
Hodge module of vanishing cycles $\HV_{U,f}^{\bu}$.

\label{sm2def5}
\end{dfn}

The (compactly-supported) hypercohomology $\bH^*(\PV_{U,f}^\bu),
\bH^*_{\rm cs}(\PV_{U,f}^\bu)$ from \eq{sm2eq1} is an important
invariant of $U,f$. If $A$ is a field then the isomorphism
$\si_{U,f}$ in \eq{sm2eq6} implies that $\bH^k(\PV_{U,f}^\bu)
\cong\bH^{-k}_{\rm cs}(\PV_{U,f}^\bu)^*$, a form of Poincar\'e
duality.

\smallskip

We defined $\smash{\PV_{U,f}^\bu}$ in perverse sheaves over a base
ring $A$. Writing $\PV_{U,f}^\bu(A)$ to denote the base ring, one
can show that $\PV_{U,f}^\bu(A)\cong \PV_{U,f}^\bu(\Z)\otL_\Z A$.
Thus, we may as well take $A=\Z$, or $A=\Q$ if we want $A$ to be a
field, since the case of general $A$ contains no more information.

\smallskip

There is a `Thom--Sebastiani Theorem for perverse sheaves', due to
Massey \cite{Mass1} and Sch\"urmann \cite[Cor.~1.3.4]{Schu}. Applied
to $\PV_{U,f}^\bu$, it yields:

\begin{thm} Let\/ $U,V$ be complex manifolds and\/ $f:U\ra\C,$
$g:V\ra\C$ be holomorphic, so that\/ $f\boxplus g:U\t V\ra\C$ is regular
with\/ $(f\boxplus g)(u,v):=f(u)+g(v)$. Set\/ $X=\Crit(f)$ and\/
$Y=\Crit(g)$ as complex analytic spaces of\/ $U,V,$ so that\/
$\Crit(f\boxplus g)=X\t Y$. Then there is a natural isomorphism
\e
\smash{\TS_{U,f,V,g}:\PV_{U\t V,f\boxplus g}^\bu\longra
\PV_{U,f}^\bu\boxtL \PV_{V,g}^\bu}
\label{sm2eq8}
\e
in $\Perv(X\t Y),$ such that the following diagrams commute:
\ea
\begin{gathered}
{}\!\!\!\!\!\!\!\!\!\!\!\!\!\!\! \xymatrix@!0@C=58pt@R=50pt{
*+[r]{\PV_{U\t V,f\boxplus g}^\bu} \ar[rrrrr]_{\si_{U\t V,f\boxplus
g}} \ar[d]^(0.4){\TS_{U,f,V,g}} &&&&&
*+[l]{\bD_{X\t Y}(\PV_{U\t V,f\boxplus g}^\bu)} \\
*+[r]{\raisebox{25pt}{\quad\,\,\,$\begin{subarray}{l}\ts\PV_{U,f}^\bu\boxtL \\
\ts\PV_{V,g}^\bu\end{subarray}$}}
\ar[rr]^(0.57){\si_{U,f}\boxtL\si_{V,g}} &&
*+[r]{\raisebox{25pt}{${}\,\,\,\begin{subarray}{l}\ts
\bD_X(\PV_{U,f}^\bu)\boxtL\!\!\!\!\!{} \\
\ts \bD_Y(\PV_{V,g}^\bu)\end{subarray}$}} \ar[rrr]^(0.33)\cong &&&
*+[l]{\bD_{X\t Y}\bigl(\PV_{U,f}^\bu\boxtL \PV_{V,g}^\bu\bigr),\!\!{}}
\ar[u]^{\bD_{X\t Y}(\TS_{U,f,V,g})} }\!\!\!\!\!{}
\end{gathered}
\label{sm2eq9}
\\[-10pt]
\begin{gathered}
{}\!\!\!\!\xymatrix@!0@C=290pt@R=40pt{ *+[r]{\PV_{U\t V,f\boxplus
g}^\bu} \ar[r]_{\tau_{U\t V,f\boxplus g}} \ar[d]^{\TS_{U,f,V,g}} &
*+[l]{\PV_{U\t V,f\boxplus g}^\bu} \ar[d]_{\TS_{U,f,V,g}} \\
*+[r]{\PV_{U,f}^\bu\boxtL \PV_{V,g}^\bu} \ar[r]^{\tau_{U,f}\boxtL\tau_{V,g}}
& *+[l]{\PV_{U,f}^\bu\boxtL \PV_{V,g}^\bu.\!\!{}} }\!\!\!\!\!{}
\end{gathered}
\label{sm2eq10}
\ea
\label{sm2thm5}
\end{thm}

Finally, we introduce some notation for pullbacks of $\PV_{V,g}^\bu$ by
local biholomorphisms.

\begin{dfn} Let $U,V$ be complex manifolds, $\Phi:U\ra V$ an
\'etale morphism, and $g:V\ra\C$ an analytic function. Write
$f=g\ci\Phi:U\ra\C$, and $X=\Crit(f)$, $Y=\Crit(g)$ as
$\C$-submanifolds of $U,V$. Then $\Phi\vert_X:X\ra Y$ is a local biholomorphism.
Define an isomorphism
\e
\PV_\Phi:\PV_{U,f}^\bu\longra\Phi\vert_X^*\bigl(\PV_{V,g}^\bu
\bigr)\quad \text{in $\Perv(X)$}
\label{sm2eq14}
\e
by the commutative diagram for each $c\in f(X)\subseteq g(Y)$:
\e
\begin{gathered}
\xymatrix@C=145pt@R=17pt{
*+[r]{\PV_{U,f}^\bu\vert_{X_c}\!=\!\phi_{f-c}^p(A_U[\dim U])
\vert_{X_c}} \ar[r]_(0.53)\al \ar@<2ex>[d]^{\PV_\Phi\vert_{X_c}} &
*+[l]{\phi_{f-c}^p\ci\Phi^*(A_V[\dim V]))\vert_{X_c}}
\ar@<-2ex>[d]_\be \\
*+[r]{\Phi\vert_{X_c}^*\bigl(\PV_{V,g}^\bu\bigr)} \ar@{=}[r] &
*+[l]{\Phi_0^*\ci\phi_{g-c}^p\ci(A_V[\dim V]))\vert_{X_c}.\!\!{}}}\!\!\!\!{}
\end{gathered}
\label{sm2eq15}
\e
Here $\al$ is $\phi_{f-c}^p$ applied to the canonical isomorphism
$A_U\ra\Phi^*(A_V)$, noting that $\dim U=\dim V$ as $\Phi$ is
a local biholomorphism.
By naturality of the isomorphisms $\al,\be$ in \eq{sm2eq15} we find
the following commute, where $\si_{U,f},\tau_{U,f}$ are as
in~\eq{sm2eq6}--\eq{sm2eq7}:
\ea
\begin{gathered}
\xymatrix@!0@C=140pt@R=35pt{ *+[r]{\PV_{U,f}^\bu}
\ar[rr]_{\si_{U,f}} \ar[d]^{\PV_\Phi} &&
*+[l]{\bD_X(\PV_{U,f}^\bu)} \\
*+[r]{\Phi\vert_X^*\bigl(\PV_{V,g}^\bu\bigr)}
\ar[r]^{\Phi\vert_X^*(\si_{V,g})} &
{\Phi\vert_X^*\bigl(\bD_Y(\PV_{V,g}^\bu)\bigr)} \ar[r]^(0.37)\cong &
*+[l]{\bD_X\bigl(\Phi\vert_X^*(\PV_{V,g}^\bu)\bigr),\!\!{}}
\ar[u]^{\bD_X(\PV_\Phi)} }
\end{gathered}
\label{sm2eq16}\\
\begin{gathered}
\xymatrix@!0@C=280pt@R=35pt{ *+[r]{\PV_{U,f}^\bu}
\ar[r]_{\tau_{U,f}} \ar[d]^{\PV_\Phi} &
*+[l]{\PV_{U,f}^\bu} \ar[d]_{\PV_\Phi} \\
*+[r]{\Phi\vert_X^*(\PV_{V,g}^\bu)} \ar[r]^{\Phi\vert_X^*(\tau_{V,g})}
& *+[l]{\Phi\vert_X^*(\PV_{V,g}^\bu).\!\!{}} }
\end{gathered}
\label{sm2eq17}
\ea

If $U=V$, $f=g$ and $\Phi=\id_U$
then~$\PV_{\id_U}=\id_{\PV_{U,f}^\bu}$.
If $W$ is another complex manifold, $\Psi:V\ra W$ is a local biholomorphism, and
$h:W\ra\C$ is analytic with $g=h\ci\Psi:V\ra\C$, then composing
\eq{sm2eq15} for $\Phi$ with $\Phi\vert_{X_c}^*$ of \eq{sm2eq15} for
$\Psi$ shows that
\e
\PV_{\Psi\ci\Phi}=\Phi\vert_X^*(\PV_\Psi)\ci\PV_\Phi:
\PV_{U,f}^\bu\longra(\Psi\ci\Phi)\vert_X^*\bigl(\PV_{W,h}^\bu\bigr).
\label{sm2eq18}
\e
That is, the isomorphisms $\PV_\Phi$ are functorial.
\label{sm2def6}
\end{dfn}

\smallskip

We conclude by saying that because of the Riemann--Hilbert correspondence, all our results on
perverse sheaves of vanishing cycles on complex
analytic spaces over a well-behaved base
ring $A$, translate immediately when $A=\C$ to the corresponding
results for $\cD$-modules of vanishing cycles, with no extra work.
and also to mixed Hodge modules on complex analytic spaces,
see \cite[\S 2.9-2.10]{BBDJS}.

\subsection{Stabilizing perverse sheaves of vanishing
cycles}
\label{s2.2}

To set up notation for Theorem
\ref{sm5thm2} below, we need the following theorem, which is proved
in Joyce~\cite[Prop.s 2.15, 2.16 \& 2.18]{Joyc1}.

\begin{thm}[Joyce \cite{Joyc1}] Let\/ $U,V$ be complex manifolds,
$f:U\ra\C,$ $g:V\ra\C$ be holomorphic, and\/ $X=\Crit(f),$ $Y=\Crit(g)$
as complex analytic subspaces of\/ $U,V$. Let\/ $\Phi:U\hookra V$ be a closed
embedding of\/ complex manifolds with\/ $f=g\ci\Phi:U\ra\C,$ and suppose
$\Phi\vert_X:X\ra V\supseteq Y$ is an isomorphism $\Phi\vert_X:X\ra
Y$. Then:
\smallskip

\noindent{\bf(i)} For each $x\in X\subseteq U$ there exist open $U'\subseteq U$ and $V'\subseteq V$ with $x\in U'$ and $\Phi(U')\subseteq V',$  an open neighbourhood $T$ of $0\in\C^n$ where $n=\dim V-\dim U,$ and a biholomorphism $\al\t\be : V' \ra U'\t T,$ such that
\begin{equation*}
(\al\t\be)\ci\Phi\vert_{U'}=\id_{U'}\t 0:U'\longra U'\t T, 
\end{equation*}
and $g\vert_{V'}=f\ci\al +(z_1^2+\cdots+z_n^2)\ci\be:V'\ra\C$. Thus, setting $f'=f\vert_{U'}:U'\ra\C,$ $g'=g\vert_{V'}:V'\ra\C,$ $\Phi'=\Phi\vert_{U'}:U'\ra V'$, $X'=\Crit(f')\subseteq U',$ and $Y'=\Crit(g')\subseteq V',$ then $f'=g'\ci\Phi':U'\ra\C,$ and $\Phi'\vert_{X'}:X'\ra Y',$ $\al\vert_{Y'}:Y'\ra X'$ are biholomorphisms.
\smallskip

\noindent{\bf(ii)} Write\/ $N_{\sst UV}$ for the normal bundle of\/
$\Phi(U)$ in $V,$ regarded as a vector bundle on $U$ in
the exact sequence of vector bundles on $U\!:$
\e
\xymatrix@C=20pt{ 0 \ar[r] & TU \ar[rr]^(0.4){\d\Phi} && \Phi^*(TU)
\ar[rr]^(0.55){\Pi_{\sst UV}} && N_{\sst UV} \ar[r] & 0.}
\label{sm5eq2}
\e
Then there exists a unique $q_{\sst UV}\in H^0(S^2N_{\sst
UV}^*\vert_X)$ which is a nondegenerate quadratic form on $N_{\sst
UV}\vert_X,$ such that whenever $U',V',\Phi',\be,n,X'$
are as in {\bf(i)\rm,} writing $\langle\d z_1,\ldots,\d
z_n\rangle_{U'}$ for the trivial vector bundle on $U'$ with basis
$\d z_1,\ldots,\d z_n,$ there is a natural isomorphism
$\hat\be:\langle\d z_1,\ldots,\d z_n\rangle_{U'}\ra N_{\sst
UV}^*\vert_{U'}$ making the following diagram commute:
\begin{gather}
\begin{gathered}
\xymatrix@C=130pt@R=15pt{
*+[r]{N_{\sst UV}^*\vert_{U'}} \ar[r]_(0.3){\Pi_{\sst UV}^*\vert_{U'}}
 & *+[l]{\Phi^*(T^*V)\vert_{U'}=\Phi^{\prime
*}(T^*V\vert_{V'})} \ar[d]_{\Phi^{\prime *}} \\
*+[r]{\langle\d z_1,\ldots,\d z_n\rangle_{U'}=\Phi^{\prime *}
\ci\be^*(T_0^*\C^n)} \ar[r]^(0.7){\Phi^{\prime *}(\d\be^*)}
\ar@{.>}[u]_{\hat\be} & *+[l]{\Phi^{\prime *}(T^*V'),}}
\end{gathered}
\label{sm5eq3}\\
\text{and\/}\qquad q_{\sst UV}\vert_{X'}=(S^2\hat\be)
\vert_{X'}(\d z_1\ot\d z_1+\cdots+\d z_n\ot\d z_n).
\label{sm5eq4}
\end{gather}

\noindent{\bf(iii)} Now suppose $W$ is another complex manifold,
$h:W\ra\C$ is holomorphic, $Z=\Crit(h)$ as a complex analytic subspace of\/ $W,$
and\/ $\Psi:V\hookra W$ is a closed embedding of complex analytic subspaces
with\/ $g=h\ci\Psi:V\ra\C$ and\/ $\Psi\vert_Y:Y\ra Z$ an
isomorphism. Define $N_{\sst VW},q_{\sst VW}$ and\/ $N_{\sst
UW},q_{\sst UW}$ using $\Psi:V\hookra W$ and\/ $\Psi\ci\Phi:U\hookra
W$ as in {\bf(ii)} above. Then there are unique morphisms $\ga_{\sst
UVW},\de_{\sst UVW}$ which make the following diagram of vector
bundles on $U$ commute, with straight lines exact:
\e
\begin{gathered}
\xymatrix@!0@C=19pt@R=9pt{ &&&&&&&&&&&&&&& 0 \ar[ddl] \\
&&&&&&&&&&&&&&&& 0 \ar[dll] \\
&&&&&& 0 \ar[dddrr] &&&&&&&& TU \ar[dddllllll]_{\d\Phi}
\ar[ddddddllll]^(0.4){\d(\Psi\ci\Phi)} \\ \\ \\
&&&&&&&& \Phi^*(TV) \ar[dddllllll]_{\Pi_{\sst UV}}
\ar[dddrr]^(0.6){\Phi^*(\d\Psi)} \\ \\
0 \ar[drr] \\
&& N_{\sst UV} \ar[dll]  \ar@{.>}[dddrrrrrr]_{\ga_{\sst UVW}}
&&&&&&&& (\Psi\ci\Phi)^*(TW) \ar[ddddddrrrr]^{\Phi^*(\Pi_{\sst VW})}
\ar[dddll]^(0.4){\Pi_{\sst UW}} \\
0 \\ \\
&&&&&&&& N_{\sst UW}  \ar[dddll] \ar@{.>}[dddrrrrrr]_{\de_{\sst
UVW}} \\ \\ \\
&&&&&& 0  &&&&&&&& \Phi^*(N_{\sst VW}) \ar[drr] \ar[ddr] \\
&&&&&&&&&&&&&&&& 0 \\
&&&&&&&&&&&&&&& 0.}\!\!\!\!\!\!\!\!\!\!\!\!\!\!\!\!\!\!\!\!\!{}
\end{gathered}
\label{sm5eq5}
\e

Restricting to $X$ gives an exact sequence of vector bundles:
\e
\xymatrix@C=14pt{ 0 \ar[r] & N_{\sst UV}\vert_X
\ar[rrr]^(0.48){\ga_{\sst UVW}\vert_X} &&& N_{\sst UW}\vert_X
\ar[rrr]^(0.45){\de_{\sst UVW}\vert_X} &&& \Phi\vert_X^*(N_{\sst
VW}) \ar[r] & 0.}
\label{sm5eq6}
\e
Then there is a natural isomorphism of vector bundles on $X$
\e
N_{\sst UW}\vert_X\cong N_{\sst UV}\vert_X\op \Phi\vert_X^*(N_{\sst
VW}),
\label{sm5eq7}
\e
compatible with the exact sequence {\rm\eq{sm5eq6},} which
identifies
\e
\begin{split}
q_{\sst UW}&\cong q_{\sst UV}\op \Phi\vert_X^*(q_{\sst VW})\op 0
\qquad\text{under
the splitting}\\
S^2N_{\sst UW}\vert_X^*&\cong S^2N_{\sst
UV}\vert_X^*\op\Phi\vert_X^*\bigl(S^2N_{\sst VW}^*\vert_Y\bigr) \op
\bigl(N_{\sst UV}^*\vert_X\ot \Phi\vert_X^*(N_{\sst VW}^*)\bigr).
\end{split}
\label{sm5eq8}
\e

\label{sm5thm1}
\end{thm}

Following \cite[Def.s 2.19 \& 2.24]{Joyc1}, we define:

\begin{dfn} Let $U,V$ be complex manifolds, $f:U\ra\C$,
$g:V\ra\C$ be holomorphic, and $X=\Crit(f)$, $Y=\Crit(g)$ as
complex manifolds of $U,V$. Suppose $\Phi:U\hookra V$ is a closed
embedding of complex manifolds with $f=g\ci\Phi:U\ra\C$ and
$\Phi\vert_X:X\ra Y$ an isomorphism. Then Theorem \ref{sm5thm1}(ii)
defines the normal bundle $N_{\sst UV}$ of $U$ in $V$, a vector
bundle on $U$ of rank $n=\dim V-\dim U$, and a nondegenerate
quadratic form $q_{\sst UV}\in H^0(S^2N_{\sst UV}^*\vert_X)$. Taking
top exterior powers in the dual of \eq{sm5eq2} gives an isomorphism
of line bundles on $U$
\begin{equation*}
\rho_{\sst UV}:K_U\ot\La^nN_{\sst UV}^*
\,{\buildrel\cong\over\longra}\,\Phi^*(K_V),
\end{equation*}
where $K_U,K_V$ are the canonical bundles of $U,V$.

\smallskip

Write $X^\red$ for the reduced subspace of $X$. As $q_{\sst
UV}$ is a nondegenerate quadratic form on $N_{\sst UV}\vert_X,$ its
determinant $\det(q_{\sst UV})$ is a nonzero section of
$(\La^nN_{\sst UV}^*)\vert_X^{\ot^2}$. Define an isomorphism of line
bundles on~$X^\red$:
\e
J_\Phi=\rho_{\sst UV}^{\ot^2}\ci
\bigl(\id_{K_U^2\vert_{X^\red}}\ot\det(q_{\sst
UV})\vert_{X^\red}\bigr):K_U^{\ot^2}\big\vert_{X^\red}
\,{\buildrel\cong\over\longra}\,\Phi\vert_{X^\red}^*
\bigl(K_V^{\ot^2}\bigr).
\label{sm5eq9}
\e

Since principal $\Z_2$-bundles $\pi:P\ra X$ in the sense of
Definition \ref{sm2def3} are a complex analytic
topological notion, and $X^\red$ and $X$ have the same topological
space, principal
$\Z_2$-bundles on $X^\red$ and on $X$ are equivalent. Define
$\pi_\Phi:P_\Phi\ra X$ to be the principal $\Z_2$-bundle which
parametrizes square roots of $J_\Phi$ on $X^\red$. That is,
complex analytic local sections $s_\al:X\ra P_\Phi$ of $P_\Phi$
correspond to local isomorphisms $\al: K_U\vert_{X^\red}
\ra\Phi\vert_{X^\red}^*(K_V)$ on $X^\red$ with~$\al\ot\al=J_\Phi$.

\smallskip

Now suppose $W$ is another complex manifold, $h:W\ra\C$ is
holomorphic, $Z=\Crit(h)$ as a complex analytic subspace of $W$, and $\Psi:V\hookra
W$ is a closed embedding of complex manifolds with $g=h\ci\Psi:V\ra\C$
and $\Psi\vert_Y:Y\ra Z$ an isomorphism. Then Theorem
\ref{sm5thm1}(iii) applies, and from \eq{sm5eq7}--\eq{sm5eq8} we can
deduce that
\e
\begin{split}
J_{\Psi\ci\Phi}=\Phi\vert_{X^\red}^*(J_\Psi)\ci J_\Phi:
K_U^{\ot^2}\big\vert_{X^\red}\,{\buildrel\cong\over\longra}\,
(\Psi\ci\Phi)\vert_{X^\red}^*\bigl(K_W^{\ot^2}\bigr)
=\,\Phi\vert_{X^\red}^*\bigl[\Psi\vert_{Y^\red}^*
\bigl(K_W^{\ot^2}\bigr)\bigr].
\end{split}
\label{sm5eq10}
\e
For the principal $\Z_2$-bundles $\pi_\Phi:P_\Phi\ra X$,
$\pi_\Psi:P_\Psi\ra Y$, $\pi_{\Psi\ci\Phi}:P_{\Psi\ci\Phi}\ra X$,
equation \eq{sm5eq10} implies that there is a canonical isomorphism
\e
\Xi_{\Psi,\Phi}:P_{\Psi\ci\Phi}\,{\buildrel\cong\over\longra}\,
\Phi\vert_X^*(P_\Psi)\ot_{\Z_2}P_\Phi.
\label{sm5eq11}
\e
It is also easy to see that these $\Xi_{\Psi,\Phi}$ have an
associativity property under triple compositions, that is, given
another complex manifold $T$, holomorphic $e:T\ra\C$ with
$Q:=\Crit(e)$, and $\Up:T\hookra U$ a closed embedding with
$e=f\ci\Up:T\ra\C$ and $\Up\vert_Q:Q\ra X$ an isomorphism, then
\e
\begin{split}
\bigl(\id_{(\Phi\ci\Up)\vert_Q^*(P_\Psi)}&\ot\kern
.1em\Xi_{\Phi,\Up}\bigr)\ci \Xi_{\Psi,\Phi\ci\Up}=
\bigl(\Up\vert_Q^*(\Xi_{\Psi,\Phi})\ot\id_{P_\Up}\bigr)\ci
\Xi_{\Psi\ci\Phi,\Up}:\\
&P_{\Psi\ci\Phi\ci\Up}\longra
(\Phi\ci\Up)\vert_Q^*(P_\Psi)\ot_{\Z_2}
\Up\vert_Q^*(P_\Phi)\ot_{\Z_2}P_\Up.
\end{split}
\label{sm5eq12}
\e
\label{sm5def}
\end{dfn}

The reason for restricting to $X^\red$ above is the following
\cite[Prop.~2.20]{Joyc1}, whose proof uses the fact that $X^\red$ is
reduced in an essential way.

\begin{lem} In Definition\/ {\rm\ref{sm5def},} the isomorphism
$J_\Phi$ in \eq{sm5eq9} and the principal\/ $\Z_2$-bundle\/
$\pi_\Phi:P_\Phi\ra X$ depend only on $U,\ab V,\ab X,\ab Y,\ab f,g$
and\/ $\Phi\vert_X:X\ra Y$. That is, they do not depend on
$\Phi:U\ra V$ apart from\/~$\Phi\vert_X:X\ra Y$.
\label{sm5lem}
\end{lem}

Using the notation of Definition \ref{sm5def}, we can restate 
Theorem 5.4 in \cite{BBDJS}:

\begin{thm}{\bf(a)} Let\/ $U,V$ be complex manifolds, $f:U\ra\C,$
$g:V\ra\C$ be holomorphic, and\/ $X=\Crit(f),$ $Y=\Crit(g)$ as
complex analytic subspaces of\/ $U,V$. Let\/ $\Phi:U\hookra V$ be a closed
embedding of\/ complex analytic subspaces with\/ $f=g\ci\Phi:U\ra\C,$ and suppose
$\Phi\vert_X:X\ra V\supseteq Y$ is an isomorphism $\Phi\vert_X:X\ra
Y$. Then there is a natural isomorphism of perverse sheaves on
$X\!:$
\e
\Th_\Phi:\PV_{U,f}^\bu\longra\Phi\vert_X^*\bigl(\PV_{V,g}^\bu\bigr)
\ot_{\Z_2}P_\Phi,
\label{sm5eq13}
\e
where\/ $\PV_{U,f}^\bu,\PV_{V,g}^\bu$ are the perverse sheaves of
vanishing cycles from\/ {\rm\S\ref{s2.1},} and\/ $P_\Phi$ the
principal\/ $\Z_2$-bundle from Definition\/ {\rm\ref{sm5def},} and
if\/ $\cQ^\bu$ is a perverse sheaf on\/ $X$ then
$\cQ^\bu\ot_{\Z_2}P_\Phi$ is as in Definition\/
{\rm\ref{sm2def3}}. Also the following diagrams commute, where
$\si_{U,f},\si_{V,g},\tau_{U,f},\tau_{V,g}$ are as
in\/~{\rm\eq{sm2eq6}--\eq{sm2eq7}:}
\e
\begin{gathered}
\xymatrix@!0@C=185pt@R=45pt{
*+[r]{\PV_{U,f}^\bu} \ar[r]_(0.25){\Th_\Phi}
\ar@<2ex>[d]^{\si_{U,f}} &
*+[l]{\Phi\vert_X^*\bigl(\PV_{V,g}^\bu\bigr)\ot_{\Z_2}P_\Phi}
\ar[r]_(0.18){\raisebox{-11pt}{$\st\Phi\vert_X^*(\si_{V,g})\ot\id$}}
& *+[l]{\Phi\vert_X^*\bigl(\bD_Y(\PV_{V,g}^\bu)\bigr)\!\ot_{\Z_2}
\!P_\Phi} \ar@<-2ex>[d]_\cong
\\
*+[r]{\bD_X(\PV_{U,f}^\bu)} &&
*+[l]{\bD_X\bigl(\Phi\vert_X^*(\PV_{V,g}^\bu)\!\ot_{\Z_2}\!P_\Phi\bigr),}
\ar[ll]_(0.7){\bD_X(\Th_\Phi)} }\!\!\!\!\!{}
\end{gathered}
\label{sm5eq14}
\e
\e
\begin{gathered}
\xymatrix@!0@C=300pt@R=45pt{
*+[r]{\PV_{U,f}^\bu} \ar[r]_(0.45){\Th_\Phi}
\ar@<2ex>[d]^{\tau_{U,f}} &
*+[l]{\Phi\vert_X^*\bigl(\PV_{V,g}^\bu\bigr)\ot_{\Z_2}P_\Phi}
\ar@<-2ex>[d]_{\Phi\vert_X^*(\tau_{V,g})\ot\id} \\
*+[r]{\PV_{U,f}^\bu} \ar[r]^(0.45){\Th_\Phi} &
*+[l]{\Phi\vert_X^*\bigl(\PV_{V,g}^\bu\bigr)\ot_{\Z_2}P_\Phi.}
}\!\!\!\!\!{}
\end{gathered}
\label{sm5eq15}
\e

If\/ $U=V,$ $f=g,$ $\Phi=\id_U$ then $\pi_\Phi:P_\Phi\ra X$ is
trivial, and\/ $\Th_\Phi$ corresponds to $\id_{\PV_{U,f}^\bu}$ under
the natural isomorphism $\id_X^*(\PV_{U,f}^\bu)
\ot_{\Z_2}P_\Phi\cong\PV_{U,f}^\bu$.
\smallskip

\noindent{\bf(b)} The isomorphism $\Th_\Phi$ in \eq{sm5eq13} depends
only on $U,\ab V,\ab X,\ab Y,\ab f,g$ and\/ $\Phi\vert_X:X\ra Y$.
That is, if\/ $\ti\Phi:U\ra V$ is an alternative choice for $\Phi$
with\/ $\Phi\vert_X=\ti\Phi\vert_X:X\ra Y,$ then
$\Th_\Phi=\Th_{\smash{\ti \Phi}},$ noting that\/
$P_\Phi=P_{\smash{\ti\Phi}}$ by Lemma\/~{\rm\ref{sm5lem}}.

\smallskip
\noindent{\bf(c)} Now suppose $W$ is another complex manifold,
$h:W\ra\C$ is holomorphic, $Z=\Crit(h),$ and\/ $\Psi:V\hookra W$ is a
closed embedding with\/ $g=h\ci\Psi:V\ra\C$ and\/ $\Psi\vert_Y:Y\ra
Z$ an isomorphism. Then Definition\/ {\rm\ref{sm5def}} defines
principal\/ $\Z_2$-bundles $\pi_\Phi:P_\Phi\ra X,$
$\pi_\Psi:P_\Psi\ra Y,$ $\pi_{\Psi\ci\Phi}:P_{\Psi\ci\Phi}\ra X$ and
an isomorphism $\Xi_{\Psi,\Phi}$ in {\rm\eq{sm5eq11},} and part\/
{\bf(a)} defines isomorphisms of perverse sheaves
$\Th_\Phi,\Th_{\Psi\ci\Phi}$ on $X$ and\/ $\Th_\Psi$ on $Y$. Then
the following commutes in $\Perv(X)\!:$
\e
\begin{gathered}
\xymatrix@C=185pt@R=27pt{*+[r]{\PV_{U,f}^\bu}
\ar[r]_(0.45){\Th_{\Psi\ci\Phi}} \ar@<2ex>[d]^{\Th_\Phi} &
*+[l]{(\Psi\ci\Phi)\vert_X^*\bigl(\PV_{W,h}^\bu\bigr)\ot_{\Z_2}
P_{\Psi\ci\Phi}} \ar@<-2ex>[d]_{\id\ot \Xi_{\Psi,\Phi}} \\
*+[r]{\Phi\vert_X^*\bigl(\PV_{V,g}^\bu\bigr)\!\ot_{\Z_2}\!P_\Phi}
\ar[r]^(0.33){\raisebox{6pt}{$\st\Phi\vert_X^*(\Th_\Psi)\ot\id$}} &
*+[l]{\Phi\vert_X^*\!\ci\!\Psi\vert_Y^*\bigl(\PV_{W,h}^\bu\bigr)
\!\ot_{\Z_2}\!\Phi\vert_X^*(P_\Psi)\!\ot_{\Z_2}\!P_\Phi.}
}\!\!\!\!\!{}
\end{gathered}
\label{sm5eq16}
\e
\label{sm5thm2}
\end{thm}

\subsection{Complex Lagrangian intersections in complex symplectic manifolds}
\label{s2.3}

We will start with a basic definition to fix the notation:

\begin{dfn} 
Let $(S, \om)$ be a {\it symplectic manifold}, i.e., a complex manifold $S$
endowed with a closed non-degenerate holomorphic $2$-form $\om\in\Omega^2_S$.
Denote the complex dimension of $S$ by $2n.$

\smallskip

A complex submanifold $M \subset S$ is {\it Lagrangian} if 
the restriction of the symplectic form $\om$ on $S$ to a $2$-form on $M$
vanishes and $\dim M = n.$

\smallskip

Holomorphic coordinates, $x_1, \ldots , x_n, y_1, .\ldots , y_n$ on
an open subset 
$S'\subset S$
in the complex analytic topology, are called {\it Darboux
coordinates} if 
$
\om=\sum_{i=1}^n \d y_i \wedge \d x_i.
$
\label{basic}
\end{dfn}

\begin{dfn} 
Given an $n$-dimensional manifold $N$, let us denote by $S=T^*N$ its cotangent bundle.  
For any chosen point $p\in U\subset N,$ for $U$ an open subset of $N$ containing $x$, 
let us denote by $(x_1,\ldots,x_n)$ a set of coordinates. Then for any $x\in U,$ the differentials
$(\d x_1)_x, \ldots, (\d x_n)_x$ form a basis of $T^*_x N.$ 

\smallskip

Namely, if $y \in T^*_x N$ then $y=
\sum_{i=1}^n y_i (\d x_i)_x$ for some complex coefficients $y_1,\ldots, y_n.$ This induces a set of coordinates 
$(x_1, \ldots, x_n,y_1,\ldots, y_n)$
on $T^*U,$ so a coordinate chart for $T^*N,$ induced by $(x_1,\ldots,x_n)$ on $U$.
It is well known that transition functions on the overlaps are holomorphic and this gives 
the structure of a complex manifold of dimension $2n$ to $T^*N$. 

\smallskip

Next, one can define a $2$-form on $T^*U$ by 
$\om=\sum_{i=1}^n \d x_i \wedge \d y_i.$ It is easy to check that the definition is coordinate-independent. 
Define the $1$-form $\al=\sum_{i=1}^n  y_i \wedge \d x_i.$ Clearly $\om = -\d \al,$ and $\al$ is intrinsically defined. The $1$-form $\al$ is called in literature the {\it Liouville form}, and the $2$-form $\om$ is the {\it canonical symplectic form}. 
\end{dfn}

Next, we will review symmetric obstruction theories on Lagrangian intersections from \cite{BeFa}, and we state a crucial definition for our program.

\smallskip

Let $(S,\om)$ be a complex symplectic manifold as above, and $L,M\subseteq S$ be Lagrangian
submanifolds. Let $X=L\cap M$ be the intersection as a complex analytic space. Then $X$
carries a canonical symmetric obstruction theory 
$\varphi:E^\bu \ra \bL_X$
in the sense of \cite{BeFa}, which can be represented by the complex $E^\bu\simeq[T^*S\vert_X\ra T^*L\vert_X\op T^*M\vert_X]$
with $T^*S\vert_X$ in degree $-1$ and $T^*L\vert_X\op T^*M\vert_X$
in degree zero. Hence
\begin{equation}
\det(E^\bu)\cong K_S\vert_X^{-1}\ot K_L\vert_X\ot K_M\vert_X
\cong K_L\vert_X\ot K_M\vert_X,
\label{detL}
\end{equation}
since $K_S\cong\O_S$. 
This motivates the following: 

\begin{dfn} 
We define an {\it orientation} of a complex Lagrangian submanifold $L$ to be a choice of square root line bundle 
$K_L^{1/2}$ for $K_L.$ 
\label{or.1}
\end{dfn}

\begin{rem}
The previous definition is inspired by \cite{BBDJS} and close to `orientation data' in Kontsevich and Soibelman \cite{KoSo1}.
We point out that {\it spin structure} could have been a better choice of name than orientation, but we use orientations for consistency with \cite{BBDJS,BBJ,BJM,Joyc1}. Also, for {\it real} Lagrangians, a square root $K_L^{1/2}$ induces an orientation on $L$ in the usual sense.
\label{spin}
\end{rem}

Now, we recall well known established results in complex symplectic geometry which will be used 
to prove our main results. We start with the {\it complex Darboux theorem}. 

\begin{thm} 
Let $(S,\om)$ be a complex symplectic manifold. Then locally in the complex analytic topology around a point
$p\in S$ is always possible to choose holomorphic Darboux coordinates.
\label{darboux}
\end{thm}

So, basically, every symplectic manifold $S$ is locally isomorphic to the cotangent bundle $T^*N$ of a manifold $N.$ The fibres of the induced vector bundle structure on $S$ are Lagrangian submanifolds, so complex analytically locally defining on $S$ a foliation by Lagrangian submanifolds, i.e., a {\it polarization}:

\begin{dfn}
A {\it polarization} of a symplectic manifold $(S,\om)$ is a holomorphic Lagrangian fibration $\pi : S' \ra E,$ where $S'\subseteq S$ is open. 
\label{polar}
\end{dfn}

Note that it is always possible to choose locally near a point $p\in S$ in the complex analytic topology Darboux coordinates $(x_1,\ldots ,x_n,y_1,\ldots,y_n)$ and compatible coordinates $x_i$ on $E$ such that $\pi$ can be identified with the projection  $(x_1,\ldots ,x_n,y_1,\ldots,y_n) \to (x_1,\ldots,x_n).$

\smallskip

We will usually consider polarizations which are {\it transverse} to the Lagrangians whose intersection we wish to study. The point of the transversality condition is that 
we have a canonical one-to-one correspondence between sections $s$ of the polarization 
$E$ for the symplectic manifold $(S,\om)$, such that $s^* (\om)=0$ and Lagrangian submanifolds of $S$ transverse to $E.$ Moreover, in terms of coordinates, near every point of 
a Lagrangian
$M\subset S$ transverse to the polarization $E$ there exists a set of Darboux coordinates
$(x_1,\ldots ,x_n,y_1,\ldots,y_n)$ such that $M=\{y_1=\cdots = y_n=0\},$ $E=\langle \d x_1, \ldots , \d x_n\rangle$
and the Euler form $s$ of $M$ inside $E$ is given by 
$s=\sum_i y_i \d x_i.$

\smallskip

If $L,M$ are complex Lagrangian submanifolds in $(S,\om)$, and we consider the projection $$\pi:(x_1,\ldots ,x_n,y_1,\ldots,y_n) \to (x_1,\ldots,x_n)$$ defining local coordinates on $L,$ then we can always assume to choose such coordinates $x_i,y_i$ transverse to $L,M$ at a point, and transverse to other coordinate systems too. In other words, we are using the projection $\pi$ as a polarization, and we assume that the leaves are transverse to the two Lagrangians, so that $L$ and $M$ turn into the graphs of $1$-forms on $N.$ The Lagrangian condition implies that these $1$-forms on $N$ are closed.

\smallskip

Recall now the {\it Lagrangian neighbourhood theorem}:

\begin{thm} 
If $M \subset (S,\om)$ is a complex Lagrangian submanifold, then there exists a complex analytic neighbourhood $V\subset M$ of a point $p\in M$ isomorphic as a complex symplectic manifold to a neighbourhood $U$ of $p$ in $T^*M,$
and $M$ is identified with the zero section in $T^*M.$
\label{lag_ngb}
\end{thm}

Note that $(S,\omega)$ need not be isomorphic to $T^*M$ in a neighbourhood of $M,$
but just in a neighbourhood of a point $p\in M.$ 

\smallskip

So, we may assume that one of these $1$-forms is the zero section of $T^*N$, 
hence identify locally $M$ with $N.$ 
By making $L$ smaller if necessary, we may assume that the closed $1$-form defined by $L$ is exact, and $L$ is the graph of the $1$-form $\d f,$ for a holomorphic function $f$ defined locally on some open submanifold $M'\subset M,$ as the following lemma states:

\begin{lem} 
Choose locally near a point $p\in S$ in the complex analytic topology Darboux coordinates $(x_1,\ldots ,x_n,y_1,\ldots,y_n)$ and compatible coordinates $x_i$ on $E$ such that $\pi:S\ra E$ can be identified with the projection  $(x_1,\ldots ,x_n,y_1,\ldots,y_n) \to (x_1,\ldots,x_n).$
Now, given a polarization $(x_1,\ldots ,x_n,y_1,\ldots,y_n) \ra (x_1,\ldots,x_n)$ defining local coordinates on $L,$
then $L$ is given by $$\Bigl\{\Bigl(x_1,\ldots ,x_n, \frac{\d f}{\d x_1},\ldots, \frac{\d f}{\d x_n}\Bigr):\; x_1,\ldots ,x_n \in M'\Bigr\}$$ for a holomorphic function $f(x_1,\ldots, x_n)$ defined locally on an open $M'\subset M\subset S,$ where $M$ is the Lagrangian submanifold identified with the zero section, and the polarization $\pi$ with projection $T^*M \ra M$. 
\end{lem}

So, in conclusion, if $\pi : S' \ra E$ is a polarization, and $M$ a Lagrangian submanifold with $\pi : M \ra E$ transverse near $x$ in $M,$ then locally there is a unique isomorphism $S' \cong T^*M$ identifying $M$ with zero section and $\pi$ with projection $T^*M \ra M$. Then any other Lagrangian $L$ in $S$ transverse to $\pi$ is locally described by the graph of $\d f$, for a holomorphic function $f$ locally defined on $M.$ It is now straightforward to deduce that, as $M$ is graph of $0$, and $L$ is graph of $\d f,$ then the intersection $X=L\cap M$ is the critical locus $(\d f)^{-1}(0).$

\medskip

We can summarize in this way. 
Let $(S,\om)$ be a complex symplectic manifold, and $L,M\subseteq S$ be Lagrangian
submanifolds. Let $X=L\cap M$ be the intersection as a complex analytic space. Then $X$
is complex analytically locally modeled on the zero locus of the $1$-form $\d f,$ that is 
on the critical locus $\Crit(f:U\ra\C)$ for a holomorphic 
function $f$ on a smooth manifold $U.$
 So, $X$ carries a natural perverse sheaf of vanishing cycles $\PV^\bu_{U,f}$ in the notation of \S\ref{s2.1},
and a natural problem to investigate is the following. Given open ${R}_i,{R}_j\subseteq S$ with isomorphisms ${R}_i\cong\Crit(f_i)$, ${R}_j\cong\Crit(f_j)$ for holomorphic
$f_i:U_i\ra\C$ and $f_j:U_j\ra\C$, we have to understand whether the
perverse sheaves $\cP^\bu_{{R}_i}=\PV_{U_i,f_i}^\bu$ on ${R}_i$ and $\cP^\bu_{{R}_j}=\PV_{U_j,f_j}^\bu$ on ${R}_j$
are isomorphic over ${R}_i\cap{R}_j$, and if so, whether
the isomorphism is canonical, for only then can we hope to glue the
$\cP^\bu_{{R}_i}$ for $i\in I$ to make $\cP^\bu_{L,M}$. 

\medskip

We will develop this program in \S\ref{s3}.

\clearpage

\section{Constructing canonical global perverse sheaves on Lagrangian intersections}
\label{s3}

We can state our main result.

\begin{thm} Let\/ $(S,\om)$ be a complex symplectic manifold and\/
$L,M$ oriented complex Lagrangian submanifolds in $S,$ and write $X=L\cap M,$
as a complex analytic subspace of\/ $S$. Then we may define
$P_{L,M}^\bu\in\Perv(X),$ uniquely up to canonical isomorphism, and
isomorphisms 
\e
\Si_{L,M}:P_{L,M}^\bu\ra \bD_X(P_{L,M}^\bu),
\qquad
\Tau_{L,M}:P_{L,M}^\bu\ra P_{L,M}^\bu,
\label{sm6eq5}
\e
respectively the Verdier duality and the monodromy
isomorphisms. These $P_{L,M}^\bu\in\Perv(X),\Si_{L,M},\Tau_{L,M}$ are locally characterized 
by the following property.

\smallskip

Given a choice of local Darboux coordinates $(x_1,\ldots,x_n,y_1,\ldots,y_n)$ 
in the sense of Definition \ref{basic} such that
$L$ is locally identified in coordinates with the graph $\Gamma_{\d f(x_1,\ldots,x_n)}$ of $\d f$ for $f$ a holomorphic function defined locally on an open $U\subset \C^n$, and $M$  is locally identified in coordinates with the graph  $\Gamma_{\d g(x_1,\ldots,x_n)}$ of $\d g$ for $g$ 
 a holomorphic function defined locally on $U$, and
the orientations $K_L^{1/2},K_M^{1/2}$ are the trivial square roots of
$K_L \cong \langle \d x_1\wedge \cdots\wedge\d x_n\rangle \cong K_M$, 
then there is a canonical isomorphism 
$P_{L,M}^\bu\cong \PV^\bu_{U,g-f},$ where $\PV^\bu_{U,g-f}$
is the perverse sheaf of vanishing cycles of $g-f,$ and $\Sigma_{L,M}$ and $\Tau_{L,M}$ are respectively the Verdier duality $\sigma_{U,g-f}$ and the monodromy $\tau_{U,g-f}$
introduced in \S\ref{s2}.

\smallskip

The same applies for $\cD$-modules and mixed Hodge modules on\/~$X$.
\label{s1thm1}
\end{thm}

A convenient way to express this is in terms of {\it charts}, by which we mean a set of data locally defined by the choice of a polarization $\pi: S \to E$. Charts are analogous to {\it critical charts} defined by \cite[\S 2]{Joyc1}, as in \S\ref{s4.1}. 
We will show in \S\ref{s4.2} that they are actually critical charts and they define the structure of a 
d-critical locus on the Lagrangian intersection, but for this section we will not use it.

\smallskip

We explained in \S\ref{s2.3}
that the local choice of a polarization on $(S,\om)$ yields a local description 
of the Lagrangian intersection as a critical locus $P \cong\Crit(f)$ for 
a closed embedding $i:P\hookra U$, $P$ open in $X,$ and  
a holomorphic function 
$f:U\subset L \ra \C,$ where $U$ is an open submanifold of $L.$ We have a local symplectic identification $S\cong T^*U \subseteq T^*L,$ which identifies $L\subset S$ with the zero section in $T^*L,$ and $M\subset S$ with the graph $\Ga_{\d f}$ of $\d f,$ and $\pi : S \ra E\cong L$ with the projection $T^*L\ra L.$ So, for each polarization $\pi: S \to E$
we have naturally induced a set of data $(P,U,f,i),$ which we will call an $L$-{\it chart}. 
We will also consider $M$-{\it charts}, namely charts coming from polarizations that identify the other Lagrangian $M$ with the zero section,
that is, charts like $(Q,V,g,j)$ where $Q \cong\Crit(g)$ for 
a closed embedding $j:Q\hookra V$, $Q$ open in $X,$ and  
a holomorphic function 
$g:V\subset M \ra \C,$ 
where $V$ is an open submanifold of $M.$ We have a local symplectic identification $S\cong T^*V \subseteq T^*M,$ which identifies $M\subset S$ with the zero section in $T^*M,$ and $L\subset S$ with the graph $\Ga_{\d g}$ of $\d g,$ and $\pi : S \ra E\cong M$ with the projection $T^*M\ra M.$ We will use also $LM$-{\it charts}.

\smallskip

Using this general technique, let us fix the following notation we will use for the rest of the paper. 
We will consider mainly three kinds of {\it charts}, where by charts we basically mean a set of data associated to a choice of one or two polarizations for our symplectic complex manifold $(S,\om)$:

\begin{enumerate}

\item[$(a)$] $L$-{\it charts} $(P,U,f,i)$ are induced by a polarizations $\pi : S' \ra E$ transverse to $L,M$ with $S'\subset S$ open,  $P\subset X$ open, and $U\subset L$ open, and $f : U \ra \C$ holomorphic, and $i : P \hookra U \subset L$ the inclusion, with $i: P\ra \Crit(f)$ an isomorphism, so that we have local identifications
\begin{itemize}
\item $(S,\om) = T^*U;$
\item $L=$ zero section; 
\item $M = \Gamma_{\d f};$
\item $E = L = U;$
\item $\pi : S \ra E$ with $\pi : T^*U \ra U.$
\end{itemize}

\item[$(b)$]  $M$-{\it charts} $(Q,V,g,j)$ are induced by a polarization $\tilde\pi : \tilde S \ra F$ transverse to $L,M,$ with $\tilde S\subset S$ open, $Q\subset X$ open, and $V\subset M$ open, and $g : V \ra\C$ holomorphic, and $j : Q \hookra V \subset M$ the inclusion, with $j: Q\ra \Crit(g)$ an isomorphism, so that we have local identifications
\begin{itemize}
\item $(S,-\om) = T^*V;$ 
\item $M=$ zero section; 
\item $L = \Gamma_{\d g};$  
\item $F = M = V;$
\item $\tilde\pi : S \ra F$ with $\pi : T^*V \ra V.$
\end{itemize}

\item[$(c)$]  $LM$-{\it charts} $(R,W,h,k)$ are induced by polarizations $\pi : S' \ra E,$ $\tilde\pi : \tilde S \ra F$ transverse to $L,M$ and to each other on $\hat S = S' \cap \tilde S.$ We have $W \subset L \t M$ open, $h : W \ra \C$ holomorphic, $k : R \ra W \subset L \t M$ the diagonal map, 
with $k: R\ra \Crit(h)$ an isomorphism,
and local identifications
\begin{itemize}
\item $(S',\om) \t (\tilde S,-\om) = T^* W;$ 
\item $L \t M =$ zero section;
\item diagonal $\De_S = \Gamma_{\d h};$
\item $E \t F = L \t M = W;$
\item $\pi \t \tilde \pi : \hat S \ra E \t F$ with $\pi : T^*W \ra W.$
\end{itemize}

Note that $R = P\cap Q.$ These three kinds of charts will be related by Proposition \ref{local}, which will give an embedding from an open subset of an $L$-chart $(P,U,f,i)$ into an $LM$-chart $(R,W,h,k),$ and similarly
an embedding from an open subset of an $M$-chart $(Q,V,g,j)$ into $(R,W,h,k).$

\end{enumerate}

\smallskip

%Let\/ $(S,\om)$ be a complex symplectic manifold and\/
%$L,M$ complex Lagrangian submanifolds in $S,$ and write $X=L\cap M,$
%as a complex analytic subspace of\/ $S$. Suppose we are given square
%roots\/ $K_L^{1/2},K_M^{1/2}$ for $K_L,K_M$.  
%We may choose 
%a family of polarizations on $\pi_a:S_a\ra E_a$ which defines
%a family $\bigl\{(P_a,U_a,f_a,i_a):a\in A\bigr\}$ of
%charts $(P_a,U_a,f_a,i_a)$ on $X$ such that $\{P_a:a\in A\}$ is
%an analytic open cover of the complex analytic space $X$, so that $P_a \cong\Crit(f_a)$
%for holomorphic functions $f_a:U_a\ra \C,$ and open $U_a\subset L$ complex manifolds 
%(and actually Lagrangians),
%and $i_a:P_a\hookra U_a$ closed embeddings.

%\smallskip

Moreover, we will explain in \S\ref{s3.1} that the choice of a polarization $\pi:S\ra E$
naturally induces 
local biholomorphisms $L{\buildrel\cong\over\longra} M$ or $M{\buildrel\cong\over\longra} L$, and thus isomorphisms 
\e
\Theta: K_L\vert_X{\buildrel\cong\over\longra} K_M\vert_X \quad {\rm or} \quad \Xi: K_M\vert_X{\buildrel\cong\over\longra} K_L\vert_X
\label{isocanbund.1}
\e
between the canonical bundles of the Lagrangian submanifolds.
We define $\pi_{P,U,f,i}:Q_{P,U,f,i}\ra P$ to be the principal\/
$\Z_2$-bundle parametrizing local isomorphisms
\e
\vartheta: K^{1/2}_L\vert_X{\buildrel\cong\over\longra} K^{1/2}_M\vert_X \quad {\rm or} \quad \xi: K^{1/2}_M\vert_X{\buildrel\cong\over\longra} K^{1/2}_L\vert_X
\label{isocanbund}
\e
such that $\vartheta\ot \vartheta=\Theta$ or $\xi\ot \xi=\Xi.$

%In other words, using local biholomorphisms between the polarization $E$ and the Lagrangians, square roots $K_L^{1/2}$ on $L$ and $K_M^{1/2}$ on $M$ give us two possibly different square roots for $K_E.$ Define $Q_{P,U,f,i}$ to be the principal $\Z_2$ bundle of local isomorphisms between these two square roots
%$K_L^{1/2}\vert_X\cong K_M^{1/2}\vert_X$ compatible with the given isomorphisms $K_L\cong K_E \cong K_M.$
%Finally, in terms of charts, 
%$\pi_{R,U,f,i}:Q_{R,U,f,i}\ra R$ is the principal\/
%$\Z_2$-bundle parametrizing local isomorphisms
%$\al:K_{X}^{1/2}\ra i^*(K_U)\vert_{R^\red}$ with\/ $\al\ot\al=
%\io_{R,U,f,i},$ for isomorphisms $\io_{R,U,f,i}:K_{X}\ra i^*(K^{\ot 2}_U)\vert_{R^\red}$
%induced by the choice of a polarization.

\smallskip

Also, on each $L$-chart, $M$-chart, $LM$-chart, 
we have a natural perverse sheaf of vanishing cycles associated to the local description of the Lagrangian intersection as a critical locus. So we 
get a perverse sheaf of vanishing cycles $i^*(\PV^\bu_{U,f})$ on $P,$ $j^*(\PV^\bu_{V,g})$ on $Q,$ and $k^*(\PV^\bu_{W,h})$ on $R.$ 
These perverse sheaves together with
principal $\Z_2$-bundles parametrizing square roots 
of isomorphisms \eq{isocanbund.1} are the objects we want to glue.

\medskip

Then $P_{L,M}^\bu\in\Perv(X)$ is characterized by the following properties:
\begin{itemize}
\setlength{\itemsep}{0pt}
\setlength{\parsep}{0pt}
\item[{\bf(i)}] If\/ $(P,U,f,i)$ is a 
an $L$-chart, $M$-chart, or $LM$-chart, there is a natural isomorphism
\e
\om_{P,U,f,i}:P_{L,M}^\bu\vert_P\longra
i^*\bigl(\PV_{U,f}^\bu\bigr)\ot_{\Z_2} Q_{P,U,f,i},
\label{sm6eq6}.
\e
Furthermore the following commute in $\Perv(P)\!:$
\ea
&\begin{gathered} \xymatrix@!0@C=280pt@R=55pt{
*+[r]{P_{L,M}^\bu\vert_P} \ar[d]^{\Si_{L,M}\vert_P}
\ar[r]_(0.35){\om_{P,U,f,i}} &
*+[l]{i^*\bigl(\PV_{U,f}^\bu\bigr)\ot_{\Z_2} Q_{P,U,f,i}}
\ar[d]_{i^*(\si_{U,f})\ot\id_{Q_{P,U,f,i}}} \\
*+[r]{\bD_P\bigl(P_{L,M}^\bu\vert_P\bigr)} &
*+[l]{\begin{aligned}[b]
\ts i^*\bigl(\bD_{\Crit(f)}(\PV_{U,f}^\bu)\bigr)\ot_{\Z_2}
Q_{P,U,f,i}&\\
\ts \cong\bD_P\bigl(i^*(\PV_{U,f}^\bu)\ot_{\Z_2}
Q_{P,U,f,i}\bigr)&,\end{aligned}}
\ar[l]_(0.7){\bD_P(\om_{P,U,f,i})}}
\end{gathered}
\label{sm6eq7}\\
&\begin{gathered} \xymatrix@!0@C=280pt@R=55pt{
*+[r]{P_{L,M}^\bu\vert_P} \ar[d]^{\Tau_{L,M}\vert_P}
\ar[r]_(0.35){\om_{P,U,f,i}} &
*+[l]{i^*\bigl(\PV_{U,f}^\bu\bigr)\ot_{\Z_2} Q_{P,U,f,i}}
\ar[d]_{i^*(\tau_{U,f})\ot\id_{Q_{P,U,f,i}}} \\
*+[r]{P_{L,M}^\bu\vert_R} \ar[r]^(0.35){\om_{P,U,f,i}} &
*+[l]{i^*\bigl(\PV_{U,f}^\bu\bigr)\ot_{\Z_2} Q_{P,U,f,i}.} }
\end{gathered}
\label{sm6eq8}
\ea

\item[{\bf(ii)}] 
Let $\pi:S'\ra E$ and $\ti\pi:\ti S\ra F$ be polarizations transverse to $L,M,$ and transverse to each other on $S'\cap \ti S$. Then from $\pi$ we get an $L$-chart $(P,U,f,i),$ from $\ti\pi$ we get an $M$-chart $(Q,V,g,j),$ and from $\pi$ and $\ti\pi$ together we get an $LM$-chart $(R,W,h,k).$
Write $(P',U',f',i')$ for the $L$-chart determined by $\pi\vert_{S'\cap\ti S}: S'\cap\ti S\ra E,$ and $(Q',V',g',i')$ for the $M$-chart determined by $\ti\pi\vert_{S'\cap\ti S}: S'\cap\ti S\ra F.$ Then $P'\subseteq P,$ $U'\subseteq U$ are open and $f'=f\vert_{U'},$ $i'=i\vert_{P'},$ so $(P',U',f',i')$ is a {\it subchart\/} of $(P,U,f,i)$ in the sense of \S\ref{s4.1}. We write this as $(P',U',f',i')\subseteq (P,U,f,i).$ Similarly, $(Q',V',g',j')\subseteq (Q,V,g,j).$ Also $P'=Q'=R = X\cap S'\cap \ti S.$

\smallskip

In this situation, Proposition \ref{local} will show that there exist closed embeddings $\Phi : U' \hookra W$ and $\Psi : V' \hookra W$ such that 
so that $h\ci \Phi=f:U'\ra \C$ and $h\ci \Psi=g:V'\ra \C.$
Moreover $\Crit(f)\cong\Crit(h)\cong\Crit(g)$ as complex analytic spaces, and $f,h$ and $g,h$
are pairs of {\it stably equivalent} functions, as explained in \S\ref{s2.2}.
Inspired by \cite[Def. 2.18]{Joyc1}, we will say that $\Phi:(P',U',f',i')\hookra (R,W,h,k)$ is an
embedding of charts if $\Phi$ is a locally closed embedding $U'\hookra W$
of complex manifolds such that $\Phi\ci i' = k\vert_{P'} : P' \ra W$ and $f = h \ci \Phi : U' \ra \C$. 
As a shorthand we write $\Phi : (P', U', f', i') \hookra (R,W,h,k)$ to mean $\Phi$ is an embedding of $(P',U',f',i')$ in $(R,W,h,k).$ 
In brief, Proposition \ref{local} in \S\ref{s3.1} will define two embeddings of charts
$\Phi:(P',U',f',i')\hookra (R,W,h,k)$ and $\Psi:(Q',V',g',j')\hookra (R,W,h,k).$

\smallskip

Given the embedding of charts $\Phi:(P',U',f',i')\hookra (R,W,h,k),$ there is a natural
isomorphism of principal\/ $\Z_2$-bundles
\e
\La_\Phi:Q_{R,W,h,k}\vert_{P'}\,{\buildrel\cong\over\longra}\,
i^*(P_\Phi)\ot_{\Z_2} Q_{P',U',f',i'}
\label{sm6eq9}
\e
on $P',$ for  
$P_\Phi$ 
defined as follows: local isomorphisms
\e
\begin{split}
\al:K_{X}^{1/2}\vert_{P'^\red} & \ra
i^*(K_{U'})\vert_{P'^\red},\;
\be:K_{X}^{1/2}\vert_{P'^\red}\ra j'^*(K_W)\vert_{P'^\red},\; \\ &
\ga:i'^*(K_{U'})\vert_{P'^\red}\ra
j^*(K_W)\vert_{P'^\red}
\label{albega}
\end{split}
\e
with $\al\ot\al=\io_{P',U',f',i'},$
$\be\ot\be=\io_{R,W,h,k}\vert_{P'^\red},$
$\ga\ot\ga=i\vert_{P'^\red}^*(J_\Phi)$ correspond to local
sections $s_\al:P'\ra Q_{P',U',f',i'},$ $s_\be:P'\ra
Q_{R,W,h,k}\vert_{P'},$ $s_\ga:P'\ra i'^*(P_\Phi)$,
for $J_{\Phi}$ as in Definition\/ {\rm\ref{sm5def},}
and for isomorphisms $\io_{R,W,h,k}:K_{X}\ra i^*(K^{\ot 2}_W)\vert_{P'^\red}$
induced by the polarization $E_1\t E_2.$

\smallskip

Then for each embedding of charts, the following diagram commutes in $\Perv(P'),$ for
$\Th_\Phi$ as in\/~{\rm\eq{sm5eq13}:}
\e
\begin{gathered}
\xymatrix@C=220pt@R=40pt{ *+[r]{P_{L,M}^\bu\vert_{P'}}
\ar[r]_(0.38){\om_{P',U',f',i'}} \ar[d]^(0.4){\om_{R,W,h,k}\vert_{P'}}
& *+[l]{i^*\bigl(\PV_{U',f'}^\bu\bigr)\ot_{\Z_2} Q_{P',U',f',i'}}
\ar[d]_(0.4){i'^*(\Th_\Phi)\ot\id_{Q_{P',U',f',i'}}} \\
*+[r]{\begin{subarray}{l}\ts j^*\bigl(\PV_{W,h}^\bu\bigr)\vert_{P'}
\ts \ot_{\Z_2}Q_{R,W,h,k}\vert_{P'}\end{subarray}}
\ar[r]^(0.38){\quad\quad\quad \id_{j^*(\PV_{W,h}^\bu)}\ot\La_\Phi} &
*+[l]{\begin{subarray}{l}\ts i^*\bigl(\Phi\vert_X^*(\PV_{W,h}^\bu)
\ot_{\Z_2}P_\Phi\bigr)
\ts\ot_{\Z_2} Q_{P',U',f',i'}.\end{subarray}} }
\end{gathered}
\label{sm6eq10}
\e
We will have an analogous commutative diagram induced by $\Psi$ on $\Perv(Q'):$
\e
\begin{gathered}
\xymatrix@C=220pt@R=40pt{ *+[r]{P_{L,M}^\bu\vert_{Q'}}
\ar[r]_(0.38){\om_{Q',V',g',j'}} \ar[d]^(0.4){\om_{R,W,h,k}\vert_{Q'}}
& *+[l]{i^*\bigl(\PV_{V,g}^\bu\bigr)\ot_{\Z_2} Q_{Q',V',g',j'}}
\ar[d]_(0.4){j'^*(\Th_\Psi)\ot\id_{Q_{Q',V',g',j'}}} \\
*+[r]{\begin{subarray}{l}\ts j'^*\bigl(\PV_{W,h}^\bu\bigr)\vert_{Q'}
\ts \ot_{\Z_2}Q_{R,W,h,k}\vert_{Q'}\end{subarray}}
\ar[r]^(0.38){\quad\quad\quad \id_{j'^*(\PV_{W,h}^\bu)}\ot\La_\Psi} &
*+[l]{\begin{subarray}{l}\ts i^*\bigl(\Psi\vert_X^*(\PV_{W,h}^\bu)
\ot_{\Z_2}P_\Psi\bigr)
\ts\ot_{\Z_2} Q_{Q',V',g',j'}.\end{subarray}} }
\end{gathered}
\label{sm6eq10.1}
\e

Using Theorem 1.13, we get isomorphisms
$$
\al: (i'^*(\PV^\bu_{U',f'})\ot Q_{P',U',f',i'})\vert_R \cong (k^*(\PV^\bu_{W,h})\ot Q_{R,W,h,k}),
$$
$$
\be: (j'^*(\PV^\bu_{V',g'})\ot Q_{Q',V',g',j'})\vert_R \cong (k^*(\PV^\bu_{W,h})\ot Q_{R,W,h,k}).
$$
Combining these, we get an isomorphism
\e
\be^{-1}\ci \al: (i'^*(\PV^\bu_{U',f'})\ot Q_{P',U',f',i'})\vert_R \cong (j'^*(\PV^\bu_{V',g'})\ot Q_{Q',V',g',j'})\vert_R,
\label{LchMch}
\e
that is, an isomorphism of perverse sheaves from $L$-charts and $M$-charts in $\Perv(P'\cap Q').$
Later, in \S\ref{s3.2} we will involve also two other polarizations for an associativity result. More precisely, following notation of \S\ref{s3.2}
we want that if we have two $L$-charts $(P_1,U_1,f_1,i_1)$ and $(P_3,U_3,f_3,i_3)$
and two $M$-charts $(Q_2,V_2,g_2,j_2)$ and $(Q_4,V_4,g_4,j_4)$ then
\e
\begin{split}
\al_{32}\vert_Y^{-1}\ci&\be_{32}\vert_Y\ci\be_{12}\vert_Y^{-1}\ci\al_{12}\vert_Y=\al_{34}\vert_Y^{-1}\ci\be_{34}\vert_Y\ci\be_{14}\vert_Y^{-1}\ci\al_{14}\vert_Y:\\
&\bigl(\PV_{U_1,f_1}^\bu \ot_{\Z_2} Q_{P_1,U_1,f_1,i_1}\bigr)\vert_{Y}\longra \bigl(\PV_{U_3,f_3}^\bu \ot_{\Z_2} Q_{P_3,U_3,f_3,i_3}\bigr)\vert_{Y}.
\end{split}
\label{asso}
\e
\end{itemize}

\smallskip

Theorem \ref{s1thm1} will be proved in \S\ref{s3.1}--\S\ref{s3.3}. In \S\ref{s3.3}, we will provide
a descent argument, which is the most technical part of the paper. We find useful to 
outline here our method of the proof.

\smallskip

Let $\{U_a\}_{a\in
I}$ be an analytic open cover for $X=L\cap M,$ 
induced by 
polarizations $\pi_a : S_a \ra E_a$ for $a\in I$, transverse to both $L$ and $M$,
and write
$U_{ab}=U_a \cap U_b$ for $a,b\in I$.
Similarly, write $U_{abc}=U_a \cap U_b \cap U_c$ for $a,b,c\in
I$. Define $\cP_a$ to be $i_a^*(\PV^\bu_{U_a,f_a})\ot_{\Z_2} Q_{P_a,U_a,f_a,i_a}$ from the discussion above, and isomorphisms
$\ga_{ab}:\cP^\bu_a \vert_{U_{ab}}\ra \cP^\bu_b\vert_{U_{ab}}$ in
$\Perv(U_{ab})$ for all\/ $a,b\in I$ with\/ $\be_{aa}=\id$ and\/
\begin{equation*}
\ga_{bc}\vert_{U_{abc}}\ci \ga_{ab}\vert_{U_{abc}}=
\ga_{ac}\vert_{U_{abc}}:\cP_a\vert_{U_{abc}}\longra
\cP_c\vert_{U_{abc}}
\end{equation*}
in $\Perv(U_{abc})$ for all\/ $a,b,c\in I$. 

\smallskip

The construction is
independent of the choice of $\{U_a\}_{a\in I}$ 
above.
Then by Theorem \ref{sm2thm3}, there exists\/
$\cP^\bu$ in $\Perv(X),$ unique up to canonical isomorphism, with
isomorphisms $\om_a: \cP^\bu \vert_{U_{a}} \ra\cP^\bu_a$ for each\/ $a\in I,$
satisfying $\ga_{ab}\ci\om_a\vert_{U_{ab}}=\om_b\vert_{U_{ab}}:
\cP^\bu\vert_{U_{ab}}\ra\cP^\bu_b\vert_{U_{ab}}$ for all\/~$a,b\in I$,
which concludes the proof of Theorem \ref{s1thm1}.
We will carry out this program in \S\ref{s3.1}--\S\ref{s3.3}.

\bigskip

Theorem \ref{s1thm1} resolves a long-standing question in the categorification of Lagrangian intersection number:
our perverse sheaf $P^\bu_{L,M}$ categorifies Lagrangian intersection numbers, in the sense that the constructible function 
$$p\ra \sum_i (-1)^{i-2n} \dim_\C \mathbb{H}^i_{\{ p \}}(X, P^\bu_{L,M}),$$
is equal to the well known Behrend function $\nu_X$ in \cite{Behr} by construction, using the expression of the Behrend function of a critical locus in terms of the Milnor fibre, as in \cite{Behr}, and so 
\e\chi(X,\nu_X)=\sum_i (-1)^{i-2n}\dim_A \mathbb{H}^i (X,P^\bu_{L,M}),
\label{categor}
\e
for a base ring $A$, which in this case is a field.
If the intersection $X$ is compact, then $[L]\cap [M]$ is given by \eq{categor}, where $[L],[M]$ are the homology classes of $L$ and $M$ in $S$.

\smallskip

Moreover, our construction may have exciting far reaching applications in symplectic geometry and topological field theory, as discussed in \S\ref{s5}.

\subsection{Canonical isomorphism of perverse sheaves on double overlaps}
\label{s3.1}

Given a complex symplectic manifold $(S,\om)$ and Lagrangian submanifolds $L,$ $M$ in $S,$ define $X$ to be their intersection as a complex analytic space.
Using results in \S\ref{s2.3}, locally in the complex analytic topology near a point $x\in X,$ we can choose an open set $S'\subset S$ and a polarization transverse to both $L$ and $M$ such that $S'\cong T^*L$ and $M\cong \Gamma_{\d f}$ so that $X=\Crit(f)$ for a holomorphic function $f$ defined on $U=L\cap S'.$ Thus we get a perverse sheaf of vanishing cycle $\PV^\bu_{U,f}.$ In this section we will investigate how two such local descriptions are related.

\smallskip

Consider $\pi_1,\pi_2:S_1,S_2\ra E_1,E_2$ two polarizations transverse to each other and both transverse to both $L$ and $M.$ Choose open neighbourhoods $U_1$ of $X\cap S_1$ in $L\cap S_1$ with $\pi_1(U_1)\subset\pi_1(M\cap S_1)$ and $V_2$ of $X\cap S_2$ in $M\cap S_2$ with $\pi_2(V_2)\subset\pi_2(L\cap S_2)$. Then we get respectively the local identifications
\begin{equation*}
\begin{split}
U_1\cong \pi_1(U_1)\subset & E_1,\; S_1\supset \pi_1^{-1}(\pi_1(U_1))\cong T^*U_1, \; L\cap \pi_1^{-1}(\pi_1(U_1)) \cong \Gamma_0,\; \\ &  M\cap \pi_1^{-1}(\pi_1(U_1)) \cong \Gamma_{\d f_1}, \; f_1:U_1\ra\C,
\end{split}
\end{equation*}
\begin{equation*}
\begin{split}
V_1\cong \pi_2(V_2)\subset & E_2,\; \bar S_2\supset \pi_2^{-1}(\pi_2(V_2))\cong T^*V_2, \; M\cap \pi_2^{-1}(\pi_2(V_2)) \cong \Gamma_0,\; \\ & L\cap \pi_2^{-1}(\pi_2(V_2)) \cong \Gamma_{\d g_2}, \; g_2:V_2\ra\C.
\end{split}
\end{equation*}
Choose an open neighbourhood $W_{12}$ of $\{(x,x):x\in X\cap S_1\cap S_2\}$ in $U_1\t V_2$ with $(\pi_1\t\pi_2)(W_{12})\subset (\pi_1\t\pi_2)(\De_S\cap (S_1\t S_2))$. Choose open neighbourhoods $U_1'$ of $X\cap S_1\cap S_2$ in $U_1$ with $\{(l,\pi_2\vert_M^{-1}\ci\pi_2(l)):l\in U_1'\}\subset W_{12}$ and $V_2'$ of $X\cap S_1\cap S_2$ in $V_2$ with $\{(\pi_1\vert_L^{-1}\ci\pi_1(m)):m\in V_2'\}\subset W_{12}$. Then we have:

\begin{prop} In the situation above, starting from polarizations $\pi_1,\pi_2:S_1,S_2\ra E_1,E_2$ and defining $f_1:U_1\ra\C$ using $\pi_1$ and $g_2:V_2\ra\C$ using $\pi_2$ and using the given notation, there exists locally a holomorphic function $h_{12}:W_{12}\ra \C$ such that the following diagram
\e
\begin{gathered}
\xymatrix@R30pt@C110pt{U_1' \ar[r]^{\Phi_{12}=\id_{U_1'} \times \pi_1\vert_{V_2'}^{-1}} \ar[rd]^{f_1\vert_{U_1'}} & W_{12} \ar[d]^{h_{12}} & V_2' \ar[l]_{\Psi_{12}=\pi_2\vert_{U_1'}^{-1}\times \id_{V_2'}} \ar[ld]_{g_2\vert_{V_2'}} \\
& \C &}
\end{gathered}
\label{main1}
\e
is commutative, that is
\e
h_{12}(l,\pi_1\vert_{V_2'}^{-1}(l))=f_1(l),
\quad \textrm{and} \quad
h_{12}(\pi_2\vert_{U_1'}^{-1}(m),m)=g_2(m),
\label{main1.1}
\e
for every $l\in U_1',$ $m\in V_2'.$ Moreover, $\Phi_{12}=\id_{U_1'} \times \pi_1\vert_{V_2'}^{-1}$ and $\Psi_{12}=\pi_2\vert_{U_1'}^{-1}\times \id_{V_2'}$ induce isomorphisms $\Crit(h_{12})\cong\Crit(f_1)\cong\Crit(g_2)$ as complex analytic spaces locally in the complex analytic topology. In particular, from Theorem \ref{sm5thm1}, we can choose
$(z_1, \ldots, z_n)$ coordinates normal to $(\id_L\t \pi_1\vert_{V_2'}^{-1})(U_1')$ in $W_{12},$
and $(w_1, \ldots, w_n)$ coordinates normal to $(\pi_2\vert_{U_1'}^{-1}\t \id_{V_2'})(V_2')$ in $W_{12},$
under which we can write $h_{12}\cong f_1 \boxplus z_1^2 + \ldots + z_n^2$ and $h_{12}\cong g_2 \boxplus w_1^2 + \ldots +w_n^2$. Following \S\ref{s2.2}, we will say that $f_1$ and $g_2$ are both stably equivalent to $h_{12}$.
\label{local}
\end{prop}

\begin{proof}
Consider the product symplectic manifold $(S\times \bar{S}, \om \op (-\om)),$ where $\bar{S}$ denotes the symplectic manifold $S$ corresponding to the symplectic form with the opposite sign.
In $S\times \bar{S}$ consider the Lagrangian submanifolds $N_1:=L\times M$ and $N_2:=\Delta_S,$
the diagonal. As explained in \S\ref{s2.3}, identify locally $(S\times \bar{S},\omega  \op -\omega)$ with $(T^*(L \t M) , \omega_{L\t M}),$ where $\omega_{L\t M}$ is the symplectic form on $T^*(L \t M),$
and thus $\pi_1\times \pi_2$ is identified with the projection $\pi:T^*(L\times M)\ra L\times M,$
that is $N_1$ with the zero section, and $N_2$ with the graph $\Gamma_{\d h_{12}}$ for a holomorphic function $h_{12}: L\times M \ra \C$ normalized by $h_{12}\vert_{(\pi_1 \times \pi_2)((L\times M)\cap \Delta_S)}=0.$
Consider the submanifold $P:=S\times M\subset S\times \bar S$ and intersect the Lagrangians $N_1$ and $N_2$ with this submanifold, yielding respectively
$N_1 \cap P= N_1$ and $N_2\cap P= \Delta_M,$ which both lie in $S\times M. $ Observe that
\e
\om\op (-\om) \vert_{P}=p_1^{-1}\om,
\label{p1om}
\e
where  $p_i:S\times S \ra S$ are the projections to the $i$-th factor. 
Consider the following diagram of inclusions and projections in $S\t\bar S$ and~$S$:
\begin{equation}
\begin{gathered}
\xymatrix@!0@C=48pt@R=20pt{  
& {N_1\cap P=N_1} \ar[dd]^(0.3){p_1} \ar@<1ex>[drr]^\subset
\\
{N_2\cap P=\De_M} \ar[rrr]_(0.55)\subset \ar[dd]^{p_1} &&& P=S\t M \ar[dd]^{p_1} \ar[rrr]_(0.45)\subset &&& *+[l]{S\t\bar S} \ar[dd]_{p_1} 
\\ 
& L \ar[drr]^\subset
\\
{M} \ar[rrr]^(0.35)\subset &&& S \ar@{=}[rrr] &&& *+[l]{S.} }
\end{gathered}
\label{eq1}
\end{equation}
Under the local symplectomorphisms $S\cong T^*U_1$ and $S\t\bar S\cong T^*(U_1\t V_2)$, equation \eq{eq1} is identified with the diagram:
\begin{equation}
\begin{gathered}
\xymatrix@!0@C=48pt@R=20pt{ 
& {z(U_1)\t z(V_2)} \ar[dd]^(0.3){\pi_{T^*U_1}} \ar@<.5ex>[drr]^\subset
\\
 {\Ga_{\d h_{12}}\vert_{(\id_{U_1}\t \pi_1\vert_{V_2'}^{-1})(U_1)}}
\ar[rrr]_(0.55)\subset \ar[dd]^{\pi_{T^*U_1}} &&& T^*U_1\t z(V_2) \ar[dd]^{\pi_{T^*U_1}} \ar[rrr]_(0.45)\subset &&& *+[l]{T^*(U_1\t V_2)} \ar[dd]_{\pi_{T^*U_1}} 
\\
& z(U_1) \ar[drr]^\subset
\\
{\Ga_{\d f_1}} \ar[rrr]^(0.35)\subset &&& T^*U_1 \ar@{=}[rrr] &&& *+[l]{T^*U_1.} }
\end{gathered}
\label{eq2}
\end{equation}
Here $z:U_1\ra T^*U_1$, $z:V_2\ra T^*V_2$ are the zero section maps. To see that $N_2\cap P=\De_{V_2}$ is identified with $\Ga_{\d h_{12}}\vert_{(\id_{U_1}\t \pi_1\vert_{V_2'}^{-1})(U_1)}$, note that $N_2$ is identified with $\Ga_{\d h_{12}}$, and
\begin{equation*}
(\pi_1\t\pi_2)(\De_{V_2})=(\pi_1\vert_{V_2'}\t\id_{V_2})(V_2)=(\id_{U_1}\t \pi_1\vert_{V_2'}^{-1})(U_1)\subset U_1\t V_2,
\end{equation*}
so that $\De_{V_2}$ is identified with a subset of $T^*(U_1\t V_2)\vert_{(\id_{U_1}\t \pi_1\vert_{V_2'}^{-1})(U_1)}\subset T^*(U_1\t V_2)$.

Equation \eq{eq2} shows that $\pi_{T^*U_1}$ maps $\Ga_{\d h_{12}}\vert_{(\id_{U_1}\t \pi_1\vert_{V_2'}^{-1})(U_1)}\ra\Ga_{\d f_1}$. Writing points of $T^*U_1$ as $(x,\al)$ for $x\in U_1$ and $\al\in T^*_xU_1$, and points of $T^*V_2$ as $(y,\be)$ for $y\in V_2$ and $\be\in T^*_xV_2$, we have
\begin{align*}
\Ga_{\d f_1}&=\bigl\{\bigl(x,\d f_1(x)\bigr):x\in U_1\bigr\},\\
\Ga_{\d h_{12}}\vert_{(\id_{U_1}\t \pi_1\vert_{V_2'}^{-1})(U_1)}&=\bigl\{\bigl(x,\d_x h_{12}(x,\pi_1\vert_{V_2'}^{-1}(x)),\pi_1\vert_{V_2'}^{-1}(x),0\bigr):x\in U_1\bigr\},
\end{align*}
where the final term $\be=0$ as $\Ga_{\d h_{12}}\vert_{(\id_{U_1}\t \pi_1\vert_{V_2'}^{-1})(U_1)}\subset T^*U_1\t z(V_2)$. The projection $\pi_{T^*U_1}:T^*U_1\t T^*V_2\ra T^*V_2$ maps $(x,\al,y,\be)\mapsto(x,\al)$. So from \eq{eq2} we see that
$\d_x h_{12}(x,\pi_1\vert_{V_2'}^{-1}(x))=\d_x f_1(x)$ for $x\in U_1$, that is, $\d \bigl(h_{12}\ci (\id_{U_1}\t\pi_1\vert_{V_2'}^{-1})\bigr)=\d f_1$ in 1-forms on $U_1$. Therefore $h_{12}\ci \bigl(\id_{U_1}\t\pi_1\vert_{V_2'}^{-1})\bigr)=f_1+c$ for some $c\in\R$. But $f_1$ and $h_{12}$ are normalized by $f_1\vert_{U_1\cap V_2}=0$ and $h_{12}\vert_{N_1\cap N_2}=0$, so as $p_1(N_1\cap N_2)\subset U_1\cap V_2$ we see that $c=0$. Hence $h_{12}\ci \bigl(\id_{U_1}\t\pi_1\vert_{V_2'}^{-1})\bigr)=f_1$, and the left hand triangle of \eq{main1}
commutes.

\smallskip

Using an analogous argument replacing \eq{eq1}--\eq{eq2} by the equations:
\begin{align*}
\xymatrix@!0@C=48pt@R=20pt{  
& {N_1\cap Q=N_1} \ar[dd]^(0.3){p_2} \ar@<1ex>[drr]^\subset
\\
{N_2\cap Q=\De_L} \ar[rrr]_(0.55)\subset \ar[dd]^{p_2} &&& Q:=L\t S \ar[dd]^{p_2} \ar[rrr]_(0.45)\subset &&& *+[l]{S\t\bar S} \ar[dd]_{p_2} 
\\ 
& M \ar[drr]^\subset
\\
{L} \ar[rrr]^(0.35)\subset &&& \bar S \ar@{=}[rrr] &&& *+[l]{\bar S,} }
\\
\xymatrix@!0@C=48pt@R=20pt{ 
& {z(U_1)\t z(V_2)} \ar[dd]^(0.3){\pi_{T^*V_2}} \ar@<.5ex>[drr]^\subset
\\
 {\Ga_{\d h_{12}}\vert_{(\pi_2\vert_{U_1'}^{-1}\t\id_{V_2})(V_2)}}
\ar[rrr]_(0.55)\subset \ar[dd]^{\pi_{T^*V_2}} &&& z(U_1)\t T^*V_2 \ar[dd]^{\pi_{T^*V_2}} \ar[rrr]_(0.45)\subset &&& *+[l]{T^*(U_1\t V_2)} \ar[dd]_{\pi_{T^*V_2}} 
\\
& z(U_1) \ar[drr]^\subset
\\
{\Ga_{\d g_2}} \ar[rrr]^(0.35)\subset &&& T^*V_2 \ar@{=}[rrr] &&& *+[l]{T^*V_2.} }
\end{align*}
we see that the right hand triangle of \eq{main1} commutes.

\smallskip

Finally, the last part of Proposition \ref{local} follows directly from Theorem \ref{sm5thm1}(i).
\end{proof}

As sketched already in \S\ref{s2},
note that the local biholomorphisms $\pi_1\vert_{V_2'}^{-1}$, $\pi_2\vert_{U_1'}^{-1}$ coming from polarizations $\pi_1,\pi_2,$ induce isomorphisms 
\eq{isocanbund.1}
between the canonical bundles of the Lagrangian submanifolds.
In terms of charts, we have an $L$-chart $(P_1,U_1,f_1,i_1),$ an $M$-chart $(Q_2,V_2,g_2,j_2)$ and an $LM$-chart $(R_{12},W_{12},h_{12},k_{12})$ induced by $E_1,E_2$, $E_1\t E_2$ respectively, where $P_1=X\cap U_1,$ $Q_2=X\cap V_2,$ $R_{12}=\{x\in X:(x,x)\in W_{12}\}$. Let us denote the corresponding principal $\Z_2$-bundles $Q_{P_1,U_1,f_1,i_1},$ $Q_{Q_2,V_2,g_2,j_2}$
and $Q_{R_{12},W_{12},h_{12},k_{12}}$
parametrizing square roots of these isomorphisms of canonical bundles as explained in 
the introduction of~\S\ref{s3}.

\smallskip

Note that Proposition \ref{local} defined two 
embeddings $\Phi_{12}:U_1'\hookra W_{12}$ and $\Psi_{12}:V_2'\hookra W_{12}$ which satisfy all the properties of Definition \ref{sm5def}, giving embeddings of charts $\Phi_{12}:(P_1',U_1',f_1',i_1')\hookra (R_{12},W_{12},h_{12},k_{12})$ and $\Psi_{12}:(Q_2',V_2',g_2',j_2')\hookra (R_{12},W_{12},h_{12},k_{12}),$
where $(P_1',U_1',f_1',i_1')$ is a subchart of $(P_1,U_1,f_1,i_1),$
and $(Q_2',V_2',g_2',j_2')$ is a subchart of $(Q_2,V_2,g_2,j_2)$ with $P_1'=Q_2'=R_{12}$.

\smallskip

Thus Definition \ref{sm5def} gives isomorphisms
of line bundles on $P^\red$:
\e
J_{\Phi_{12}}:K_{U_1}^{\ot^2}\big\vert_{P_1^{\prime\red}}
\,{\buildrel\cong\over\longra}\,\Phi_{12}\vert_{P_1^{\prime\red}}^*
\bigl(K_{W_{12}}^{\ot^2}\bigr),
\label{sm5eq9.1}
\e
induced by $\Phi_{12}:(P_1',U_1',f_1',i_1')\hookra (R_{12},W_{12},h_{12},k_{12})$, and
\e
J_{\Psi_{12}}:K_{V_2}^{\ot^2}\big\vert_{Q_2^{\prime\red}}
\,{\buildrel\cong\over\longra}\,\Psi_{12}\vert_{Q_2^{\prime\red}}^*
\bigl(K_{W_{12}}^{\ot^2}\bigr),
\label{sm5eq9.2}
\e
induced by $\Psi_{12}:(Q_2',V_2',g_2',j_2')\hookra (R_{12},W_{12},h_{12},k_{12})$.
%\smallskip 

%Observe now that $J_\Phi$ is an isomorphism $K_L^2 {\buildrel\cong\over\longra} \bigl(K_L \otimes K_M\bigr)\vert_{X^{\red}},$ and $J_\Psi$ is an isomorphism $K_M^2 {\buildrel\cong\over\longra} \bigl(K_L \otimes K_M\bigr)\vert_{X^{\red}}.$ So basically, we can say that $J_\Phi$ and $J_\Psi$ are both induced by local biholomophisms \eq{localbiholo1}--\eq{localbiholo2} induced by the polarizations. 

\medskip

Following Definition \ref{sm5def}, define
$\pi_{\Phi_{12}}:P_{\Phi_{12}}\ra P_1'$,
$\pi_{\Psi_{12}}:P_{\Psi_{12}}\ra Q_2'$ to be the principal $\Z_2$-bundles parametrizing square roots of $J_{\Phi_{12}},J_{\Psi_{12}}$ on $R_{12}^\red$. Then we naturally get isomorphisms of principal $\Z_2$-bundles $\La_\Phi$ and $\La_\Psi$
\e
\La_{\Phi_{12}}:Q_{R_{12},W_{12},h_{12},k_{12}}\,{\buildrel\cong\over\longra}\,
P_{\Phi_{12}}\ot_{\Z_2} Q_{P_1,U_1,f_1,i_1}\vert_{R_{12}},
\label{sm6eq9.1}
\e
\e
\La_{\Psi_{12}}:Q_{R_{12},W_{12},h_{12},k_{12}}\,{\buildrel\cong\over\longra}\,
P_{\Psi_{12}}\ot_{\Z_2} Q_{Q_2,V_2,g_2,j_2}\vert_{R_{12}}.
\label{sm6eq9.2}
\e

%Roughly speaking, the principal $\Z_2$ bundles $P_\Phi$ and $P_\Psi$ are the principal $\Z_2$ bundles of choices of isomorphisms
%$$K_L^{1/2} {\buildrel\cong\over\longra} K_M^{1/2}$$ which square to isomorphisms \eq{isocanbund} coming from $\pi_1,\pi_2.$

Thus, we can apply Theorem \ref{sm5thm2}, which yields 
natural isomorphisms of perverse sheaves on
$X$:
\e
\Th_{\Phi_{12}}:\PV_{U_1',f_1'}^\bu\longra\Phi_{12}\vert_{P_1'}^*\bigl(\PV_{W_{12},h_{12}}^\bu\bigr)
\ot_{\Z_2}P_{\Phi_{12}},
\label{iso1}
\e
\e
\Th_\Psi:\PV_{V_2',g_2'}^\bu\longra\Psi_{12}\vert_{Q_2'}^*\bigl(\PV_{W_{12},h_{12}}^\bu\bigr)
\ot_{\Z_2}P_{\Psi_{12}},
\label{iso2}
\e
where\/ $\PV_{U_1',f_1'}^\bu,\PV_{V_2',g_2'}^\bu,\PV_{W_{12},h_{12}}^\bu$ are the perverse sheaves of
vanishing cycles from\/ {\rm\S\ref{s2.1},} and
if\/ $\cQ^\bu$ is a perverse sheaf on\/ $X$ then
$\cQ^\bu\ot_{\Z_2}P_{\Phi_{12}}$ is as in Definition\/
{\rm\ref{sm2def3}}. Also diagrams \eq{sm5eq14} and \eq{sm5eq15} commute.
Now, combining the isomorphisms \eq{sm6eq9.1}--\eq{iso2} we get
isomorphisms
\ea
\al_{12}=\Th_{\Phi_{12}}\ot\La_{\Phi_{12}}^{-1}&:\bigl(\PV_{U_1,f_1}^\bu \ot_{\Z_2} Q_{P_1,U_1,f_1,i_1}\bigr)\vert_{R_{12}} \longra \PV_{W_{12},h_{12}}^\bu  \ot_{\Z_2} Q_{R_{12},W_{12},h_{12},k_{12}},
\label{iso3}\\
\be_{12}=\Th_{\Psi_{12}}\ot\La_{\Psi_{12}}^{-1}&:\bigl(\PV_{V_2,g_2}^\bu \ot_{\Z_2} Q_{Q_2,V_2,g_2,j_2}\bigr)\vert_{R_{12}} \longra \PV_{W_{12},h_{12}}^\bu  \ot_{\Z_2} Q_{R_{12},W_{12},h_{12},k_{12}},
\label{iso3.1}\\
\be_{12}^{-1}\ci\al_{12}&:\bigl(\PV_{U_1,f_1}^\bu \ot_{\Z_2} Q_{P_1,U_1,f_1,i_1}\bigr)\vert_{R_{12}} \longra \bigl(\PV_{V_2,g_2}^\bu \ot_{\Z_2} Q_{Q_2,V_2,g_2,j_2}\bigr)\vert_{R_{12}}.
\label{iso3.2}
\ea

\medskip

\subsection{Comparing perverse sheaves of vanishing cycles associated to polarizations}
\label{s3.2}

\medskip

Given a complex symplectic manifold $(S,\om)$ and $L,$ $M$ Lagrangian submanifolds in $S,$ define $X$ to be their intersection.
From \S\ref{s2.3}, locally in the complex analytic topology near a point $x\in X,$ we can choose an open set $S'\subset S$ and we can choose a polarization transverse to both $L$ and $M$ such that $S'\cong T^*L$ and $M\cong \Gamma_{\d f}$ so that $X\cap S'=\Crit(f)$ for a holomorphic function $f$ defined on $U\subseteq L\cap S'.$ Thus we get a perverse sheaf of vanishing cycle $\PV^\bu_{U,f}.$ In \S\ref{s3.1} we compared perverse sheaves of vanishing cycles associated to two transverse polarizations. In this section we will investigate about how they behave if we consider four polarizations, pairwise transverse in a $4$-cycle. This result will be used in \S\ref{s3.3} to prove Theorem \ref{s1thm1}.

\smallskip

We choose four polarizations $\pi_i:S\ra E_i$ for $i=1,\ldots,4$ 
all transverse to each other except perhaps for the pairs $E_1$, $E_3$ and $E_2,$ $E_4.$ Define $L$-charts $(P_1,U_1,f_1,i_1),(P_3,U_3,f_3,i_3)$ from $\pi_1,\pi_3$ and $M$-charts $(Q_2,V_2,g_2,j_2),(Q_4,V_4,g_4,j_4)$ from $\pi_2,\pi_4$, as in the beginning of \S\ref{s3}. Define $LM$-charts $(R_{12},W_{12},h_{12},k_{12})$ from $\pi_1,\pi_2$, $(R_{32},W_{32},h_{32},k_{32})$ from $\pi_3,\pi_2$, $(R_{34},W_{34},h_{34},k_{34})$ from $\pi_3,\pi_4$, 
$(R_{14},W_{14},h_{14},k_{14})$ from $\pi_1,\pi_4$ as in \S\ref{s3.1}, with embeddings of charts $\Phi_{12},\Psi_{12},\ldots,\Phi_{14},\Psi_{14}$ from subcharts of $(P_a,U_a,f_a,i_a),(Q_b,V_b,g_b,j_b)$ to $(R_{ab},W_{ab},h_{ab},k_{ab})$.

\medskip

Similarly to Proposition \ref{local}, we have the following result:

\begin{prop}
Given a complex symplectic manifold $(S,\om)$ and $L,$$M$ Lagrangian submanifolds in $S,$ define $X$ to be their intersection.
Suppose we are given four polarizations $\pi_1,\ldots,\pi_4,$ and choose data 
$(P_a,U_a,f_a,i_a)$ for $a=1,3,$ $(Q_b,V_b,g_b,j_b)$ for $a=2,4$ and\/ $(R_{ab},W_{ab},h_{ab},k_{ab}),\Phi_{ab},\Psi_{ab}$ for $ab=12,32,34,14$ as above.

\smallskip

Then there exist an open set $Z$ in $U_1\t V_2\t U_3\t V_4,$ a holomorphic function $F:Z\ra\C,$ and open neighbourhoods $U_a'$ of $X\cap S_1\cap S_2\cap S_3\cap S_4$ in $U_a,$ and\/ $V_b'$ of $X\cap S_1\cap S_2\cap S_3\cap S_4$ in $V_b,$ and\/ $W_{ab'}$ of\/ $\bigl\{(x,x):x\in X\cap S_1\cap S_2\cap S_3\cap S_4\bigr\}$ in $W_{ab}$ for all\/ $a=1,3,$ $b=2,4$ such that the following diagram commutes:

\clearpage

\e
\begin{gathered}
\text{\begin{footnotesize}$\displaystyle
\xymatrix@!0@R50pt@C93pt{ *+[r]{U_1'} \ar[dddddrrr]^(0.25){f_1\vert_{U_1'}}
\ar[ddrr]^(0.65){\begin{subarray}{l}\id_{U_1}\times \pi_{1}\vert_{M}^{-1}\t  \\ (\pi_{3}\vert_{M}\ci \pi_{1}\vert_{M}^{-1})\t \pi_{1}\vert_{M}^{-1}\end{subarray}}
\ar[rr]_{\Phi_{12}\vert_{U_1'}=\id_{U_1}\times \pi_{1}\vert_{V_2}^{-1}}
\ar[dd]^(0.6){\begin{subarray}{l}\Phi_{14}\vert_{U_1'}=\\\pi_{1}\vert_{V_4}^{-1}\times \\ \id_{U_1}\end{subarray}} && W_{12}' \ar[dddddr]^(0.2){h_{12}\vert_{W_{12}'}}
\ar[dd]_(0.25){
\begin{subarray}{l} \pi_{U_1}\times  \pi_{V_2} \times \\ 
(\pi_{3}\vert_{M} \ci \pi_{V_2})\times \\ \pi_{V_2}\end{subarray}} && 
*+[l]{V_2'}
\ar[ddll]^(0.45){\begin{subarray}{l} \pi_{2}\vert_{L}^{-1}\t 
\id_{\ti M}\times \\ \pi_{3}\vert_{\ti M}\t \id_{V_2'}\end{subarray}}
\ar[dddddl]_(0.25){g_2\vert_{V_2'}} \ar[ll]^{\Psi_{12}\vert_{V_2'}=\pi_{2}\vert_{L}^{-1}
\times id_{V_2'}}
\ar[dd]^{\begin{subarray}{l} \Psi_{32}\vert_{V_2'}= \\ id_{V_2'} \times
\\ \pi_{2}\vert_{L}^{-1}\end{subarray}}
\\
\\
*+[r]{W_{14}'} 
\ar@<.4ex>@/^.2pc/[dddrrr]_(0.55){h_{14}\vert_{W_{14}'}}
\ar[rr]^(0.5){\begin{subarray}{l}
\pi_{U_1} \t (\pi_{2}\vert_{L}\ci (\pi_{4}\vert_{L}\t \pi_{1}\vert_{M})^{-1})\t \\
(\pi_{3}\vert_{M}\ci (\pi_{4}\vert_{L}\t \pi_{1}\vert_{M})^{-1})\t \pi_{V_4'}
\end{subarray}} && {Z} \ar[dddr]^(0.5){F} && *+[l]{W_{32}'}
\ar[dddl]^(0.3){h_{32}\vert_{W_{32}'}} 
\ar[ll]_{\begin{subarray}{l} 
\pi_{U_3'}\t(\pi_2\vert_{L}\ci\pi_{U_3'})\t(\pi_3\vert_{M}\ci\pi_{V_2'})\t\pi_{V_2'}\end{subarray}}
\\
\\
*+[r]{V_4'} 
\ar[uurr]_(0.34){\begin{subarray}{l} \pi_{4}\vert_{L}^{-1}
\times (\pi_{2}\vert_{L}\ci \pi_{4}\vert_{L}^{-1})\t
\\ \pi_{4}\vert_{L}^{-1}\t \id_{V_4'}\end{subarray}}
\ar@<-.5ex>[drrr]_(0.4){g_4\vert_{V_4'}} \ar[uu]_{\begin{subarray}{l} \Psi_{14}\vert_{V_4'}= \\ id_{V_4'}\times \\
\pi_{4}\vert_{L}^{-1}\end{subarray}} \ar[rr]_(0.7){\Psi_{34}\vert_{V_4'}=\pi_{4}\vert_{L}^{-1}\times id_{V_4'}} && 
W_{34}'
\ar@<-.5ex>@/_.2pc/[dr]_(0.4){h_{34}\vert_{W_{34}'}} \ar[uu]_(0.45){\begin{subarray} \small
\pi_{U_3'}\times \\ (\pi_{2}\vert_{L}\ci \pi_{U_3'}) \t \\ \pi_{U_3'} \t  \pi_{V_4'}
\end{subarray}} && *+[l]{U_3'}
\ar[uull]_(0.6){\begin{subarray}{l} 
\id_{U_3'} \t \pi_2\vert_{L} \t \\
\id_{U_3'} \t \pi_3\vert_{M}^{-1}
\end{subarray}}
\ar@<.5ex>[dl]^(0.3){f_3\vert_{U_3'}} \ar[ll]_{\Phi_{34}\vert_{U_3'}=id_{U_3'}\times \pi_{3}\vert_{M}^{-1}}
\ar[uu]_(0.3){\begin{subarray}{l}\Phi_{32}\vert_{U_3'}=\\ \pi_3\vert_{M}^{-1} \times \\ id_{U_3'}
\end{subarray}}
\\
&&& \C.  }\!\!\!\!\!\!\!\!\!\!\!\!\!{}
$\end{footnotesize}}
\end{gathered}
\label{main2.1}
\e
Moreover, locally in the complex analytic topology $\Crit(F)\cong \Crit(h_{ij})\cong\Crit(f_i)\cong\Crit(g_j)$ as complex analytic spaces for all $i,j.$ In particular, we can choose appropriate coordinate systems under which we can write $F$ as the sum of functions $h_{ij}$ or $f_i$ or $g_j$ and non degenerate quadratic forms, that is they are all stably equivalent to each other in the sense of Theorem \ref{sm5thm1}.
\label{main2}
\end{prop}

\begin{proof} In equation \eq{main2.1} there are three kinds of small triangles:
\begin{itemize}
\setlength{\itemsep}{0pt}
\setlength{\parsep}{0pt}
\item[(i)] Eight triangles with vertices $U_a',W_{ab}',Z$ or $V_b',W_{ab}',Z$.
\item[(ii)] Eight triangles with vertices $U_a',W_{ab}',\C$ or $V_b',W_{ab}',\C$.
\item[(iii)] Four triangles with vertices $W_{ab}',Z,\C$.
\end{itemize}
To show that \eq{main2.1} commutes, we must show all these triangles commute. For the triangles of type (i) we can just check this by hand in an elementary way. The triangles of type (ii) commute by Proposition \ref{local} applied to $\pi_1,\pi_2$ or $\pi_3,\pi_2$ or $\pi_3,\pi_4$ or $\pi_1,\pi_4$. This leaves the triangles of type (iii), which we will show commute by a similar proof to Proposition~\ref{local}.

\smallskip

Consider the product symplectic manifold $(S\times \bar{S}\times S\times \bar{S}, \om\op -\om\op \om \op -\om),$ where $\bar{S}$ denotes the symplectic manifold $S$ corresponding to the symplectic form with the opposite sign. Write $p_i: S\times S\times S\times S\ra S$ for the projection to the $i$-th factor.
In $S\times \bar S\times S \times \bar S$ consider the Lagrangian submanifolds $N_1:=L\times \Delta_{S}\times M$ and $N_2:=\Delta_{S}\times \Delta_{S}.$
Identify it with $T^*(L\times M\times L\times M)$, and thus $\pi_1\times \pi_2\times \pi_3\times \pi_4$ with $\pi:T^*(L\times M\times L\times M)\ra L\times M\times L\times M,$ that is $N_1$ with the zero section, and $N_2:=\Gamma_{\d F}$ for a holomorphic function 
$F:Z\ra\C$ for open $Z\subseteq L\times M\times L\times M$ normalized by $F\vert_{(\pi_1 \times \pi_2\times \pi_3\times \pi_4)(N_1\cap N_2)}=0.$
Consider the submanifolds
$$P_{12}:= S\times S\times \De_{S}^{34}, \quad P_{32}:= L\times S\times S\times M, \quad P_{34}:=  \De_{S}^{12}\times S\times S, \quad P_{14}:= S\times \De_{S}^{32}\times S.$$

In the same style as the proof of Proposition \ref{main1}, intersect the Lagrangians with these submanifolds. We can either identify $N_1$ with the zero section and $N_2=\Ga_{\d F},$
or $N_1=\Ga_{-\d F}$ and $N_2$ with the zero section. We will use both the options.
Let us start with the submanifold $P_{12},$ for which we use the second identification. 
Consider the following diagram of inclusions and projections in $S\t\bar S\t S\t\bar S$ and~$S\t \bar S$:
\begin{equation}
\begin{gathered}
\xymatrix@!0@C=60pt@R=20pt{ 
& {N_1\cap P_{12}=L \t\De^{234}_{M}} \ar[dd]^(0.3){p_1\t p_2} \ar[drr]^\subset
\\
{N_2\cap P_{12}=N_2} \ar[rrr]_(0.55)\subset \ar[dd]^{p_1\t p_2} &&& P_{12}=S\times S\times \De_{S}^{34} \ar[dd]^{p_1\t p_2} \ar[rrr]_(0.45)\subset &&& *+[l]{S\t\bar{S} \t S\t\bar{S}} \ar[dd]_{p_1\t p_2} 
\\
& {L\t M} \ar[drr]
\\
{\De_{S}} \ar[rrr]^(0.55)\subset &&& S\t\bar{S} \ar@{=}[rrr] &&& *+[l]{S\t\bar{S}.} }
\end{gathered}
\label{eq1.1}
\end{equation}
Under the local symplectomorphisms $S\t \bar S\cong T^*(L\t M)$ and $S\t\bar S\t S\t\bar S\cong T^*(L\t M\t L\t M)$, equation \eq{eq1.1} is identified with the diagram:
\begin{equation}
\begin{gathered}
\xymatrix@!0@C=60pt@R=30pt{  
& {\Ga_{-\d F}\vert_{(\pi_{L}\times  \pi_{M} \times 
(\pi_{3}\vert_{M} \ci \pi_{M})\times \pi_{M})(W_{12}')}} \ar[drr]^(0.6)\subset \ar[dd]^(0.35){\pi_{12}}
\\
{z(L)\!\t\! z(M)\!\t \!z(L)\!\t\! z(M)}
\ar[rrr]_(0.50)\subset \ar[dd]^{\pi_{12}} &&& z(L)\!\t\! z(M)\!\t\! T^*L\!\t\! T^*M \ar[dd]^{\pi_{12}} \ar[rrr]_(0.35)\subset &&& *+[l]{T^*L\!\t\! T^*M \!\t\! T^*L\!\t\! T^*M} \ar[dd]_{\pi_{12}} 
\\
& {\Ga_{-\d h_{12}}} \ar[drr]^\subset
\\
{z(L)\t z(M)} \ar[rrr]^(0.55)\subset &&& T^*L\t T^*M \ar@{=}[rrr] &&& *+[l]{T^*L\t T^*M.} }
\end{gathered}
\label{eq2.1}
\end{equation}
Here $z:L\ra T^*L$, $M\ra T^*M$ are the zero section maps. To understand this, note that $\pi_1\t\pi_2\t\pi_2\t\pi_4$ maps 
$N_1\cap P_{12}=L \t\De^{234}_{M}$ to the submanifold $(\pi_{L}\times  \pi_{M} \times (\pi_{3}\vert_{M} \ci \pi_{M})\times \pi_{M})(L\t M)$ in $L\t M\t L\t M$. Our identification $S\t\bar S\t S\t\bar S\cong T^*(L\t M\t L\t M)$ maps $N_1\mapsto\Ga_{-\d F}$. Hence the top term $N_1\cap P_{12}$ in \eq{eq1.1} is identified with the top term $\Ga_{-\d F}\vert_{(\pi_{L}\times  \pi_{M} \times (\pi_{3}\vert_{M} \ci \pi_{M})\times \pi_{M})(W_{12}')}$ in \eq{eq2.1}. As for \eq{eq1}--\eq{eq2}, we see from \eq{eq1.1}--\eq{eq2.1} that the triangle of type (iii) in \eq{main2.1} with vertices the top centre $L\t M$, and $L\t M\t L\t M$, and $\C$, commutes.

\smallskip

Similarly, taking intersections with the submanifold $P_{14}$ gives a diagram analogous to \eq{eq1.1}:
$$
\begin{gathered}
\xymatrix@!0@C=60pt@R=20pt{  
& {N_1\cap P_{14}=N_1} \ar[dd]^(0.3){p_1\t p_4} \ar[drr]^\subset
\\
{N_2\cap P_{14}=\De_{S}^{1234}} \ar[rrr]_(0.55)\subset \ar[dd]^{p_1\t p_4} &&& P_{14}= S \t \De_{S}^{23}\t S  \ar[dd]^{p_1\t p_4} \ar[rrr]_(0.45)\subset &&& *+[l]{S\t\bar{S} \t S\t\bar{S}} \ar[dd]_{p_1\t p_4} 
\\
& {L\t M} \ar[drr]^\subset
\\
{\De_{S}} \ar[rrr]^(0.55)\subset &&& S\t \bar{S} \ar@{=}[rrr] &&& *+[l]{S\t\bar{S}.} }
\end{gathered}
$$
Using the first identification, this is identified with the diagram
$$
\begin{gathered}
\xymatrix@!0@C=60pt@R=30pt{ 
& z(L)\t z(M)\t z(L)\t z(M) \ar[drr]^\subset \ar[dd]^(0.3){\pi_{14}}
\\
{\Ga_{\d F}\vert_{\begin{subarray}{l}(\pi_{L} \t (\pi_{2}\vert_{L}\ci (\pi_{4}\vert_{L}\t \pi_{1}\vert_{M})^{-1})\t \\
(\pi_{3}\vert_{M}\ci (\pi_{4}\vert_{L}\t \pi_{1}\vert_{M})^{-1})\t \pi_{M})(W_{14}')\end{subarray}}}
\ar[rrr]_(0.50)\subset \ar[dd]^{\pi_{14}} &&&    z(L)\t T^*M\t T^*L\t z(M) \ar[dd]^{\pi_{14}} \ar[rrr]_(0.45)\subset &&& *+[l]{T^*(L\t M\t L \t M)} \ar[dd]_{\pi_{14}} 
\\
& {z(L)\t z(M)} \ar[drr]^\subset
\\
{\Ga_{\d h_{14}}} \ar[rrr]^(0.55)\subset &&& T^*(L\t M) \ar@{=}[rrr] &&& *+[l]{T^*(L\t M).} }
\end{gathered}
$$
Here $\pi_1\t\pi_2\t\pi_2\t\pi_4$ maps $N_2\cap P_{14}=\De_S^{1234}$ to $(\pi_{L} \t (\pi_{2}\vert_{L}\ci (\pi_{4}\vert_{L}\t \pi_{1}\vert_{M})^{-1})\t (\pi_{3}\vert_{M}\ci (\pi_{4}\vert_{L}\t \pi_{1}\vert_{M})^{-1})\t \pi_{M})(L\t M)$, and we identify $N_2$ with $\Ga_{\d F}$, which is how we get the first term on the middle line. From this we see that the triangle of type (iii) in \eq{main2.1} with vertices the left hand $L\t M$, and $L\t M\t L\t M$, and $\C$, commutes. The remaining two type (iii) triangles can be shown to commute in a similar way. Hence \eq{main2.1} commutes.
Finally, the last part of Proposition \ref{main2} follows directly from Theorem \ref{sm5thm1}(i).
\end{proof}

In the situation of Proposition \ref{main2}, set $Y=X\cap S_1\cap S_2\cap S_3\cap S_4$. Then following the reasoning of \eq{sm5eq9.1}--\eq{iso3.2} which defined the isomorphisms of perverse sheaves $\al_{12},\be_{12}$ in \eq{iso3}--\eq{iso3.1}, from \eq{main2.1} we get a commutative diagram of isomorphisms of perverse sheaves:
\e
\begin{gathered}
\text{\begin{footnotesize}$\displaystyle
\xymatrix@!0@R50pt@C93pt{ *+[r]{\begin{subarray}{l}\ts\bigl(\PV_{U_1,f_1}^\bu \ot_{\Z_2} \\ \ts Q_{P_1,U_1,f_1,i_1}\bigr)\vert_{Y}\end{subarray}} \ar[ddrr]\ar[rr]_{\al_{12}\vert_Y}
\ar[dd]^(0.6){\al_{14}\vert_Y} && 
{\begin{subarray}{l}\ts \bigl(\PV_{W_{12},h_{12}}^\bu  \ot_{\Z_2} \\ \ts Q_{R_{12},W_{12},h_{12},k_{12}}\bigr)\vert_Y\end{subarray}} \ar[dd] && 
*+[l]{\begin{subarray}{l}\ts \bigl(\PV_{V_2,g_2}^\bu \ot_{\Z_2} \\ \ts Q_{Q_2,V_2,g_2,j_2}\bigr)\vert_{Y}\end{subarray}}
\ar[ddll]
\ar[ll]^{\be_{12}\vert_Y}
\ar[dd]_{\be_{32}\vert_Y}
\\
\\
*+[r]{\begin{subarray}{l}\ts \bigl(\PV_{W_{14},h_{14}}^\bu  \ot_{\Z_2} \\ \ts Q_{R_{14},W_{14},h_{14},k_{14}}\bigr)\vert_Y\end{subarray}} 
\ar[rr] && {\PV_{Z,F}^\bu\ot_{\Z_2}Q_{Z,F}}  && *+[l]{\begin{subarray}{l}\ts \bigl(\PV_{W_{32},h_{32}}^\bu  \ot_{\Z_2} \\ \ts Q_{R_{32},W_{32},h_{32},k_{12}}\bigr)\vert_Y\end{subarray}}
\ar[ll]
\\
\\
*+[r]{\begin{subarray}{l}\ts \bigl(\PV_{V_4,g_4}^\bu \ot_{\Z_2} \\ \ts Q_{Q_4,V_4,g_4,j_4}\bigr)\vert_{Y}\end{subarray}} 
\ar[uurr] \ar[uu]_{\be_{14}\vert_Y} \ar[rr]^{\be_{34}\vert_Y} && 
{\begin{subarray}{l}\ts \bigl(\PV_{W_{34},h_{34}}^\bu  \ot_{\Z_2} \\ \ts Q_{R_{34},W_{34},h_{34},k_{34}}\bigr)\vert_Y\end{subarray}}
\ar[uu] && *+[l]{\begin{subarray}{l}\ts \bigl(\PV_{U_3,f_3}^\bu \ot_{\Z_2} \\ \ts Q_{P_3,U_3,f_3,i_3}\bigr)\vert_{Y}\end{subarray}}
\ar[uull]
\ar[ll]_{\al_{34}\vert_Y}
\ar[uu]^{\al_{32}\vert_Y}
  }\!\!\!\!\!\!\!\!\!\!\!\!\!{}
$\end{footnotesize}}
\end{gathered}
\label{main2.2}
\e
Since \eq{main2.2} commutes, we deduce that
\e
\begin{split}
\al_{32}\vert_Y^{-1}\ci&\be_{32}\vert_Y\ci\be_{12}\vert_Y^{-1}\ci\al_{12}\vert_Y=\al_{34}\vert_Y^{-1}\ci\be_{34}\vert_Y\ci\be_{14}\vert_Y^{-1}\ci\al_{14}\vert_Y:\\
&\bigl(\PV_{U_1,f_1}^\bu \ot_{\Z_2} Q_{P_1,U_1,f_1,i_1}\bigr)\vert_{Y}\longra \bigl(\PV_{U_3,f_3}^\bu \ot_{\Z_2} Q_{P_3,U_3,f_3,i_3}\bigr)\vert_{Y}.
\end{split}
\label{main2.3}
\e

Equation \eq{main2.3} tells us something important. Suppose we start with polarizations $\pi_1:S_1\ra E_1$ and $\pi_3:S_3\ra E_3$ transverse to $L,M$, and use them to define $L$-charts $(P_1,U_1,f_1,i_1)$ and $(P_3,U_3,f_3,i_3)$, and hence perverse sheaves $\PV_{U_1,f_1}^\bu \ot_{\Z_2} Q_{P_1,U_1,f_1,i_1}$ on $P_1=X\cap S_1$ and $\PV_{U_3,f_3}^\bu \ot_{\Z_2} Q_{P_3,U_3,f_3,i_3}$ on $P_3=X\cap S_3$. We wish to relate these perverse sheaves on the overlap $X\cap S_1\cap S_3$. To do this, we choose another polarization $\pi_2:S_2\ra E_2$ transverse to $L,M,\pi_1,\pi_3$, and define an $M$-chart $(Q_2,V_2,g_2,j_2)$ and $LM$-charts $(R_{12},W_{12},h_{12},k_{12})$ and $(R_{32},W_{32},h_{32},k_{32})$. Then as in \eq{iso3.2}, $\al_{32}^{-1}\ci\be_{32}\ci\be_{12}^{-1}\ci\al_{12}$ provides the isomorphism 
$\PV_{U_1,f_1}^\bu \ot_{\Z_2} Q_{P_1,U_1,f_1,i_1}\cong\PV_{U_3,f_3}^\bu \ot_{\Z_2} Q_{P_3,U_3,f_3,i_3}$ we want on $X\cap S_1\cap S_2\cap S_3$. Equation \eq{main2.3} shows that this isomorphism is independent of the choice of polarization $\pi_2:S_2\ra E_2$.

\subsection{Descent for perverse sheaves}
\label{s3.3}

To conclude the proof of Theorem \ref{s1thm1}, we use Theorem \ref{sm2thm3}, 
so in particular a descent argument to glue and get a global perverse sheaf. In this 
section we adopt the point of view of charts induced by polarizations. 
This proof follows similar ideas to \cite[\S 6.3]{BBDJS}.

\medskip

Let\/ $(S,\om)$ be a complex symplectic manifold and\/
$L,M$ complex Lagrangian submanifolds in $S,$ and write $X=L\cap M,$
as a complex analytic subspace of\/ $S$. Suppose we are given square
roots\/ $K_L^{1/2},K_M^{1/2}$ for $K_L,K_M$.  
We may choose 
a family of polarizations $\pi_a:S_a\ra E_a$ which defines
a family $\bigl\{(R_a,U_a,f_a,i_a):a\in A\bigr\}$ of
$L$-charts $(P_a,U_a,f_a,i_a)$ on $X$ such that $\{P_a:a\in A\}$ is
an analytic open cover of the analytic space $X$, so that $P_a \cong\Crit(f_a)$
for holomorphic functions $f_a:U_a\ra \C,$ and $U_a$ complex manifolds (Lagrangians),
and $i_a:P_a\hookra U_a$ closed embeddings.  

\smallskip

Then for each $a\in A$
we have a perverse sheaf
\e
i_a^*\bigl(\PV_{U_a,f_a}^\bu\bigr)\ot_{\Z_2}
Q_{P_a,U_a,f_a,i_a}\in\Perv(R_a),
\label{sm6eq14}
\e
for $Q_{P_a,U_a,f_a,i_a}$ the principal $\Z_2$ bundle 
defined in \S\ref{s3} point $(i)$
parametrizing choices of square roots of canonical bundles
$K_L^{1/2} {\buildrel\cong\over\longra} K_M^{1/2}$ which square to isomorphisms \eq{isocanbund.1}.
As explained already in the introduction of \S\ref{s3}, the idea
of the proof is to use Theorem \ref{sm2thm3}(ii) to glue the
perverse sheaves \eq{sm6eq14} on the analytic open cover $\{P_a:a\in
A\}$ to get a global perverse sheaf $P_{L,M}^\bu$ on $X$.

\smallskip

We already know from Proposition \ref{local} that, given an $L$-chart $(P,U,f,i)$ and an $M$-chart $(Q,V,g,j)$
we have the isomorphism \eq{LchMch}, which we recall here:
$$
\xymatrix{\be^{-1}\ci\al: (i^*(\PV^\bu_{U,f})\ot Q_{P,U,f,i})\vert_R \ar[r]^{\quad\cong} & (j^*(\PV^\bu_{V,g})\ot Q_{Q,V,g,j})\vert_R,}
$$
that is, an isomorphism of perverse sheaves from $L$-charts and $M$-charts in $\Perv(P\cap Q).$

\smallskip

Now, to develop our program, we have to show that
if $(P_a,U_a,f_a,i_a)$ and $(P_b,U_b,f_b,i_b)$ are $L$-charts, then we have a canonical isomorphism
\e
\xymatrix{
\de_{ab} : 
(i_a^*(\PV^\bu_{U_a,f_a})\ot Q_{P_a,U_a,f_a,i_a})\vert_{P_a\cap P_b} \ar[r]^{\quad\cong} &
(i_b^*(\PV^\bu_{U_b,f_b})\ot Q_{P_b,U_b,f_b,i_b})\vert_{P_a\cap P_b}}
\label{beta.1}
\e
with the property that for any $M$-chart $(Q,V,g,j)$ coming from $\tilde \pi : \tilde S \ra F$ transverse to $\pi_a$ and $\pi_b,$ we have
\e
\de_{ab} \vert_{P_a\cap P_b\cap Q} = \alpha_{U_b,f_b,W',h'}^{-1}\vert_{P_a\cap P_b\cap Q}\ci\be_{W',h',V,g}\vert_{P_a\cap P_b\cap Q}\ci\be_{W,h,V,g}^{-1}\vert_{P_a\cap P_b\cap Q}\ci \al_{U_a,f_a,W,h} \vert_{P_a\cap P_b\cap Q}.
\label{beta}
\e

\smallskip

%for all $a,b\in A$ we have to construct isomorphisms in $\Perv(P_a\cap
%P_b)$
%\e
%\begin{split}
%\al_{ab}:\bigl[i_a^*\bigl(\PV_{U_a,f_a}^\bu\bigr)\ot_{\Z_2}
%Q_{R_a,U_a,f_a,i_a}\bigr] \big\vert_{R_a\cap R_b}\longra
%\bigl[i_b^*\bigl(\PV_{U_b,f_b}^\bu\bigr)\ot_{\Z_2}
%Q_{R_b,U_b,f_b,i_b}\bigr] \big\vert_{R_a\cap R_b} \end{split}
%\label{sm6eq15}
%\e
%satisfying $\al_{aa}=\id$ for all $a\in A$ and
%\e
%\al_{bc}\vert_{R_a\cap R_b\cap R_c}\ci \al_{ab}\vert_{R_a\cap
%R_b\cap R_c}=\al_{ac}\vert_{R_a\cap R_b\cap R_c}\quad\text{for all
%$a,b,c\in A$.}
%\label{sm6eq16}
%\e

To prove this, we first use Proposition \ref{main2}, which provides an associativity result as in \eq{asso} or \eq{main2.3}. In particular, it shows that if $(Q',V',g',j')$ is 
another such $M$-chart, then
\begin{equation*}
\begin{split}
&\alpha_{U_b,f_b,W',h'}^{-1} \vert_{P_a\cap P_b\cap Q\cap Q'} \ci
\be_{W',h',V,g} \vert_{P_a\cap P_b\cap Q\cap Q'} \ci
\be^{-1}_{W,h,V,g} \vert_{P_a\cap P_b\cap Q\cap Q'}\ci
\alpha_{U_a,f_a,W,h} \vert_{P_a\cap P_b\cap Q\cap Q'} = \\&
\alpha_{U_b,f_b,W'',h''}^{-1} \vert_{P_a\cap P_b\cap Q\cap Q'} \ci
\be_{W'',h'',V',g'} \vert_{P_a\cap P_b\cap Q\cap Q'} \ci
\be^{-1}_{W''',h''',V',g'} \vert_{P_a\cap P_b\cap Q\cap Q'} \ci
\alpha_{U_a,f_a,W''',h'''} \vert_{P_a\cap P_b\cap Q\cap Q'}.  
\end{split}
\end{equation*}

Fix two charts $(P_a,U_a,f_a,i_a)$ and $(P_b,U_b,f_b,i_b),$
and choose a 
family $\bigl\{(Q_c,V_c,g_c,j_c):c\in I\bigr\}$ of
$M$-charts $(Q_c,V_c,g_c,j_c)$ on $X$ 
transverse to both $(P_a,U_a,f_a,i_a)$ and $(P_b,U_b,f_b,i_b),$
such that $\{Q_c:c\in I\}$ is
an analytic open cover of $P_a\cap P_b$.
Then, we can use the sheaf property of morphisms of perverse sheaves 
in the sense of Theorem \ref{sm2thm3}, 
to get $\de_{ab}$ as in \eq{beta} by gluing
$$\alpha_{U_b,f_b,W',h'}^{-1} \vert_{P_a\cap P_b\cap Q} \ci
\be_{W',h',V,g}\vert_{P_a\cap P_b\cap Q} \ci
\be^{-1}_{W,h,V,g} \vert_{P_a\cap P_b\cap Q}\ci
\alpha_{U_a,f_a,W,h} \vert_{P_a\cap P_b\cap Q}$$  on the open cover $\{Q_c:c\in I\}$.
Also, $\de_{ab}$ satisfy \eq{beta} for all $(Q,V,g,j)$, and this is
independent of choice of $(Q,V,g,j)$. This is because we can run the construction above with the family
$\bigl\{(P_a,U_a,f_a,i_a): a\in A\bigr\}\amalg\bigl\{(Q,V,g,j)\bigr\},$
yielding the same result.

\smallskip

Moreover, on $P_a\cap P_b\cap P_c$ we have
$\de_{bc}\ci \de_{ab} = \de_{ac}.$
This is because, given locally a polarization $\tilde \pi : \tilde S \ra F$ transverse to all of $\pi_a,$ $\pi_b,$ $\pi_c,$ then on $P_a \cap P_b \cap P_c \cap Q,$ we can easily check that
\begin{equation*}
\begin{split}
\ga_{bc}\ci \ga_{ab}\vert_{P_a\cap P_b\cap P_c\cap Q} =& ( \alpha_{U_c,f_c,W'',h''}^{-1} \ci \be_{W'',h'',V,g} \ci 
\be^{-1}_{W',h',V,g} \ci  \alpha_{U_b,f_b,W',h'}) \ci  \\&
(\alpha_{U_b,f_b,W',h'}^{-1} \ci \be_{W',h',V,g}^{-1} \ci
\be^{-1}_{W,h,V,g}\ci\alpha_{U_a,f_a,W,h}) =  \\ &
(\alpha_{U_c,f_c,W'',h''}^{-1} \ci \be_{W'',h'',V,g} \ci 
\be^{-1}_{W,h,V,g}\ci
\alpha_{U_a,f_a,W,h}) 
= \ga_{ac}\vert_{P_a\cap P_b\cap P_c\cap Q}.
\end{split}
\end{equation*}
As we can cover $P_a \cap P_b \cap P_c$ by such open $P_a \cap P_b \cap P_c \cap Q,$ we deduce that $\ga_{bc}\ci \ga_{ab} = \ga_{ac}$ by the sheaf property of morphisms of perverse sheaves in the sense of Theorem \ref{sm2thm3}.

\smallskip

In conclusion, we have an open cover of $X$ by $L$-charts $(P_a,U_a,f_a,i_a),$ and isomorphisms
\eq{beta.1},
satisfying $\ga_{bc}\ci \ga_{ab} = \ga_{ac}.$ So by stack property of perverse sheaves
in the sense of 
Theorem \ref{sm2thm3}(ii), 
we get that there exists $P_{L,M}^\bu$ in $\Perv(X)$, unique up to
canonical isomorphism, with isomorphisms
\begin{equation*}
\xymatrix{
\om_{P_a,U_a,f_a,i_a}:P_{L,M}^\bu\vert_{P_a}\ar[r]^{\cong\quad} &
i_a^*\bigl(\PV_{U_a,f_a}^\bu\bigr)\ot_{\Z_2} Q_{P_a,U_a,f_a,i_a}}
\end{equation*}
as in \eq{sm6eq6} for each $a\in A$, with
$\ga_{ab}\ci\om_{P_a,U_a,f_a,i_a}\vert_{P_a\cap P_b}=
\om_{P_b,U_b,f_b,i_b}\vert_{P_a\cap P_b}$ for all $a,b\in A$. Also,
\eq{sm6eq7}--\eq{sm6eq8} with $(P_a,U_a,f_a,i_a)$ in place of
$(P,U,f,i)$ define isomorphisms $\Si_{L,M}\vert_{P_a}$,
$\Tau_{L,M}\vert_{P_a}$ for each $a\in A$. The prescribed values for
$\Si_{L,M}\vert_{P_a},\Tau_{L,M}\vert_{P_a}$ and
$\Si_{L,M}\vert_{P_b},\Tau_{L,M}\vert_{P_b}$ agree when restricted
to $P_a\cap P_b$ for all $a,b\in A$. Hence, Theorem \ref{sm2thm3}(i)
gives unique isomorphisms $\Si_{L,M},\Tau_{L,M}$ in \eq{sm6eq5} such
that \eq{sm6eq7}--\eq{sm6eq8} commute with $(P_a,U_a,f_a,i_a)$ in
place of $(P,U,f,i)$ for all~$a\in A$.

\smallskip

Also, the whole construction is independent
of the choice of the family of $L$-charts and polarizations. 
This is because we can
suppose $\bigl\{(P_a,U_a,f_a,i_a):a\in A\bigr\}$ and $\bigl\{(\ti
P_a,\ti U_a,\ti f_a,\ti\imath_a):a\in\ti A\bigr\}$ are alternative
choices above, yielding $P_{L,M}^\bu,\Si_{L,M},\Tau_{L,M}$ and $\ti
P_{L,M}^\bu,\ti\Si_{L,M},\ti\Tau_{L,M}$. Then applying the same
construction to the family $\bigl\{(P_a,U_a,f_a,i_a):a\in
A\bigr\}\amalg\bigl\{(\ti P_a,\ti U_a,\ti f_a,\ti\imath_a):a\in\ti
A\bigr\}$ to get $\hat P_{L,M}^\bu$, we have canonical isomorphisms
$P_{L,M}^\bu\cong\hat P_{L,M}^\bu\cong\ti P_{L,M}^\bu$, which
identify $\Si_{L,M},\Tau_{L,M}$ with $\ti\Si_{L,M},\ti\Tau_{L,M}$.
Thus $P_{L,M}^\bu,\Si_{L,M},\Tau_{L,M}$ are independent of choices
up to canonical isomorphism.

%we get $P^\bu_{LM}$ on $X$ unique up to canonical isomorphism such that 
%$$
%P^\bu_{LM}\vert_{P_a} = (i_a^*(\PV^\bu_{U_a,f_a})\ot Q_{P_a,U_a,f_a,i_a})\vert_{P_a}.
%$$

\section{Analytic d-critical locus structure on complex Lagrangian intersections}
\label{s4}

Pantev et al.\ \cite{PTVV} show that derived intersections
$L\cap M$ of algebraic Lagrangians $L,M$ in an algebraic symplectic
manifold $(S,\om)$ have $(-1)$--shifted symplectic structures, so that
Theorem 6.6 in \cite{BBDJS} gives them the structure of algebraic
d-critical loci. Here, we will prove a complex analytic
version of this. Theorem \ref{sm6thm5}
states that the Lagrangian intersection $L\cap M$ of (oriented) complex
Lagrangians $L,M$ has the structure of an (oriented) complex
analytic d-critical locus. Notice at this point that we could have then used \cite[Thm 6.9]{BBDJS} to define a
perverse sheaf $P_{L,M}^\bu$ on $L\cap M$, instead of going through 
Theorem \ref{s1thm1} in \S\ref{s3}, but we wanted to provide a clear and direct proof 
about how to glue perverse sheaves on complex Lagrangian intersections in a 
complex analytic setup, and using only classical and symplectic geometry.
Note also that we cannot prove Theorem \ref{sm6thm5} by going via \cite{PTVV}, as they do not do a complex analytic version.

\smallskip

Here is the result of the section.

\begin{thm} Suppose $(S,\om)$ is a complex
symplectic manifold, and\/ $L,M$ are (oriented) complex Lagrangian submanifolds
in $S$. Then the intersection $X=L\cap M,$ as a complex analytic
subspace of\/ $S,$ extends naturally to a (oriented) complex analytic
d-critical locus\/ $(X,s)$. The canonical bundle $K_{X,s}$ 
in the sense of Theorem 3.5 in \S\ref{s4.1} is naturally isomorphic
to\/~$K_L\vert_{X^\red}\ot K_M\vert_{X^\red}$.
\label{sm6thm5}
\end{thm}

Theorem \ref{sm6thm5} will be proved in \S\ref{s4.2}, while in \S\ref{s4.1}
we recall some material from \cite{Joyc1}.

\subsection{Background material on d-critical loci}
\label{s4.1}

Here are some of the main definitions and results on d-critical
loci, from Joyce \cite[Th.s 2.1, 2.13, 2.21 \& Def.s 2.3, 2.11,
2.23]{Joyc1}. 

\smallskip

The key idea of this section, d-critical loci, is explained in
Definition \ref{dc2def1} below. As a preliminary, we need to
associate a sheaf $\cS_X$ to each complex analytic space $X$, such
that (very roughly) sections of $\cS_X$ parametrize different ways
of writing $X$ as $\Crit(f)$ for $U$ a complex manifold and
$f:U\ra\C$ holomorphic.

\begin{thm} Let\/ $X$ be a complex analytic space. Then there
exists a sheaf\/ $\cS_X$ of\/ $\C$-vector spaces on $X,$ unique up
to canonical isomorphism, which is uniquely characterized by the
following two properties:
\begin{itemize}
\setlength{\itemsep}{0pt}
\setlength{\parsep}{0pt}
\item[{\bf(i)}] Suppose\/ $U$ is a complex manifold, $R$ is an open
subset in\/ $X,$ and\/ $i:R\hookra U$ is an embedding of\/ $R$
as a closed complex analytic subspace of\/ $U$. Then we have an
exact sequence of sheaves of\/ $\C$-vector spaces on $R\!:$
\e
\smash{\xymatrix@C=30pt{ 0 \ar[r] & I_{R,U} \ar[r] &
i^{-1}(\O_U) \ar[r]^{i^\sh} & \O_X\vert_R \ar[r] & 0, }}
\label{dc2eq1}
\e
where $\O_X,\O_U$ are the sheaves of holomorphic functions on
$X,U,$ and\/ $i^\sh$ is the morphism of sheaves of\/
$\C$-algebras on $R$ induced by $i,$ which is surjective as $i$
is an embedding, and\/ $I_{R,U}=\Ker(i^\sh)$ is the sheaf of
ideals in $i^{-1}(\O_U)$ of functions on $U$ near $i(R)$ which
vanish on $i(R)$.

There is an exact sequence of sheaves of\/ $\C$-vector spaces on
$R\!:$
\e
\xymatrix@C=20pt{ 0 \ar[r] & \cS_X\vert_R
\ar[rr]^(0.4){\io_{R,U}} &&
\displaystyle\frac{i^{-1}(\O_U)}{I_{R,U}^2} \ar[rr]^(0.4)\d &&
\displaystyle\frac{i^{-1}(T^*U)}{I_{R,U}\cdot i^{-1}(T^*U)}\,, }
\label{dc2eq2}
\e
where $\d$ maps $f+I_{R,U}^2\mapsto \d f+I_{R,U}\cdot
i^{-1}(T^*U)$.
\item[{\bf(ii)}] Let\/ $R,U,i,\io_{R,U}$ and\/
$S,V,j,\io_{S,V}$ be as in {\bf(i)} with\/ $R\subseteq
S\subseteq X,$ and suppose $\Phi:U\ra V$ is holomorphic with\/
$\Phi\ci i=j\vert_{R}$ as a morphism of complex analytic spaces
$R\ra V$. Then the following diagram of sheaves on $R$ commutes:
\e
\begin{gathered}
\xymatrix@C=12pt{ 0 \ar[r] & \cS_X\vert_{R} \ar[d]^\id
\ar[rrr]^(0.4){\io_{S,V}\vert_{R}} &&&
\displaystyle\frac{j^{-1}(\O_{V})}{I_{S,V}^2}\Big\vert_{R}
\ar@<-2ex>[d]^{i^{-1}(\Phi^\sh)_*} \ar[rr]^(0.4)\d &&
\displaystyle\frac{j^{-1}(T^*V)}{I_{S,V}\cdot
j^{-1}(T^*V)}\Big\vert_{R} \ar@<-2ex>[d]^{i^{-1}(\d\Phi)_*} \\
 0 \ar[r] & \cS_X\vert_{R} \ar[rrr]^(0.4){\io_{R,U}} &&&
\displaystyle\frac{i^{-1}(\O_{U})}{I_{R,U}^2} \ar[rr]^(0.4)\d &&
\displaystyle\frac{i^{-1}(T^*U)}{I_{R,U}\cdot i^{-1}(T^*U)}\,.
}\!\!\!\!\!\!\!{}
\end{gathered}
\label{dc2eq3}
\e
Here $\Phi:U\ra V$ induces $\Phi^\sh:\Phi^{-1}(\O_{V})\ra\O_{U}$
on $U,$ so we have
\e
i^{-1}(\Phi^\sh):j^{-1}(\O_{V})\vert_{R}=i^{-1}\ci
\Phi^{-1}(\O_{V})\longra i^{-1}(\O_{U}),
\label{dc2eq4}
\e
a morphism of sheaves of\/ $\C$-algebras on $R$. As $\Phi\ci
i=j\vert_{R},$ equation \eq{dc2eq4} maps $I_{S,V}\vert_{R}\ra
I_{R,U},$ and so maps $I_{S,V}^2\vert_{R}\ra I_{R,U}^2$. Thus
\eq{dc2eq4} induces the morphism in the second column of\/
\eq{dc2eq3}. Similarly, $\d\Phi:\Phi^{-1}(T^*V)\ra T^*U$ induces
the third column of\/~\eq{dc2eq3}.
\end{itemize}

There is a natural decomposition\/
$\cS_X=\cSz_X\op\C_X,$ where\/ $\C_X$ is the constant sheaf on
$X$ with fibre $\C,$ and the subsheaf\/ $\cSz_X\subset\cS_X$ is
the kernel of the composition
\e
\xymatrix@C=40pt{ \cS_X \ar[r]^{\be_X} & \O_X
\ar[r]^(0.47){i_X^\sh} & \O_{X^\red},  }
\label{dc2eq5}
\e
with\/ $X^\red$ the reduced complex analytic subspace of\/ $X,$
and\/ $i_X:X^\red\hookra X$ the inclusion.
\label{dc2thm1}
\end{thm}

Thus, if we can write $X=\Crit(f)$ for $f:U\ra\C$ holomorphic, then
we obtain a natural section $s\in H^0(\cS_X)$. Essentially
$s=f+I_{\d f}^2$, where $I_{\d f}\subseteq\O_U$ is the ideal
generated by $\d f$. Note that $f\vert_X=f+I_{\d f}$, so $s$
determines $f\vert_X$. Basically, $s$ remembers all of the
information about $f$ which makes sense intrinsically on $X$, rather
than on the ambient space~$U$.

\medskip

We can now define d-critical loci:

\begin{dfn} A {\it complex analytic\/} {\it d-critical locus\/}
is a pair $(X,s)$, where $X$ is a complex analytic space, and $s\in
H^0(\cSz_X)$ for $\cSz_X$ as in Theorem \ref{dc2thm1}, satisfying
the condition that for each $x\in X$, there exists an open
neighbourhood $R$ of $x$ in $X$, a complex manifold $U$, a
holomorphic function $f:U\ra\C$, and an embedding $i:R\hookra U$ of
$R$ as a closed complex analytic subspace of $U$, such that
$i(R)=\Crit(f)$ as complex analytic subspaces of $U$,
and~$\io_{R,U}(s\vert_R)=i^{-1}(f)+I_{R,U}^2$.
We call the quadruple $(R,U,f,i)$ a {\it critical
chart\/} on~$(X,s)$.

\smallskip

Let $(X,s)$ be a complex analytic d-critical locus, and $(R,U,f,i)$ a
critical chart on $(X,s)$. Let $U'\subseteq U$ be open, and
set $R'=i^{-1}(U')\subseteq R$, $i'=i\vert_{R'}:R'\hookra U'$, and
$f'=f\vert_{U'}$. Then $(R',U',f',i')$ is a critical chart on
$(X,s)$, and we call it a {\it subchart\/} of $(R,U,f,i)$. As a
shorthand we write~$(R',U',f',i')\subseteq (R,U,f,i)$.

\smallskip

Let $(R,U,f,i),(S,V,g,j)$ be critical charts on $(X,s)$, with
$R\subseteq S\subseteq X$. An {\it embedding\/} of $(R,U,f,i)$ in
$(S,V,g,j)$ is a locally closed embedding $\Phi:U\hookra V$ such
that $\Phi\ci i=j\vert_R$ and $f=g\ci\Phi$. As a shorthand we write
$\Phi: (R,U,f,i)\hookra(S,V,g,j)$. If $\Phi:(R,U,f,i)\hookra
(S,V,g,j)$ and $\Psi:(S,V,g,j)\hookra(T,W,h,k)$ are embeddings, then
$\Psi\ci\Phi:(R,U,i,e)\hookra(T,W,h,k)$ is also an embedding.
\label{dc2def1}
\end{dfn}

\begin{thm} Let\/ $(X,s)$ be a d-critical locus, and\/
$(R,U,f,i),(S,V,g,j)$ be critical charts on $(X,s)$. Then for each\/
$x\in R\cap S\subseteq X$ there exist subcharts
$(R',U',f',i')\subseteq(R,U,f,i),$ $(S',V',g',j')\subseteq
(S,V,g,j)$ with\/ $x\in R'\cap S'\subseteq X,$ a critical chart\/
$(T,W,h,k)$ on $(X,s),$ and embeddings $\Phi:(R',U',f',i')\hookra
(T,W,h,k),$ $\Psi:(S',V',g',j')\hookra(T,W,h,k)$.
\label{da6thm2}
\end{thm}

\begin{thm} Let\/ $(X,s)$ be a complex analytic d-critical locus, and\/
$X^\red\subseteq X$ the associated reduced complex analytic space. Then there
exists a holomorphic line bundle $K_{X,s}$ on $X^\red$ which we call the
\begin{bfseries}canonical bundle\end{bfseries} of\/ $(X,s),$ which
is natural up to canonical isomorphism, and is characterized by the
following properties:
\begin{itemize}
\setlength{\itemsep}{0pt}
\setlength{\parsep}{0pt}
\item[{\bf(i)}] If\/ $(R,U,f,i)$ is a critical chart on
$(X,s),$ there is a natural isomorphism
\begin{equation*}
\io_{R,U,f,i}:K_{X,s}\vert_{R^\red}\longra
i^*\bigl(K_U^{\ot^2}\bigr)\vert_{R^\red},
\end{equation*}
where $K_U=\La^{\dim U}T^*U$ is the canonical bundle of\/ $U$ in
the usual sense.
\item[{\bf(ii)}] Let\/ $\Phi:(R,U,f,i)\hookra(S,V,g,j)$ be an
embedding of critical charts on $(X,s),$ and\/ $N_{\sst
UV},q_{\sst UV}$ be as in Proposition\/ {\rm\ref{sm5thm1}} and
set\/ $n=\dim V-\dim U$. Taking top exterior powers in the dual
of\/ \eq{sm5eq2} and pulling back to $R^\red$ using $i^*$ gives
an isomorphism of line bundles on $R^\red$
\e
\rho_{\sst UV}:i^*(K_U)\ot i^*(\La^nN_{\sst UV}^*)\vert_{R^\red}
\,{\buildrel\cong\over\longra}\,j^*(K_V)\vert_{R^\red}.
\label{dc2eq23}
\e
As $q_{\sst UV}$ is a nondegenerate quadratic form on
$i^*(N_{\sst UV}),$ its determinant $\det(q_{\sst UV})$ is a
nonvanishing section of\/ $i^*(\La^nN_{\sst UV}^*)^{\ot^2}$.
Then the following diagram of isomorphisms of line bundles on
$R^\red$ commutes:
\begin{equation*}
\xymatrix@C=150pt@R=19pt{ *+[r]{K_{X,s}\vert_{R^\red}}
\ar[r]_(0.4){\io_{R,U,f,i}} \ar[d]^{\io_{S,V,g,j}\vert_{R^\red}}
& *+[l]{i^*\bigl(K_U^{\ot^2}\bigr)\vert_{R^\red}}
\ar[d]_{\id_{i^*(K_U^2)}\ot\det(q_{\sst UV})\vert_{R^\red}} \\
*+[r]{j^*\bigl(K_V^{\ot^2}\bigr)\big\vert_{R^\red}}  &
*+[l]{i^*\bigl(K_U^{\ot^2}\bigr)\ot
i^*\bigl(\La^nN_{\sst UV}^*\bigr){}^{\ot^2}
\vert_{R^\red}.\!\!{}} \ar[l]_(0.6){\rho_{\sst UV}^{\ot^2}} }
\end{equation*}
\end{itemize}
\label{dc2thm3}
\end{thm}

\begin{dfn} Let $(X,s)$ be a complex analytic d-critical locus, and
$K_{X,s}$ its canonical bundle from Theorem \ref{dc2thm3}. An {\it
orientation\/} on $(X,s)$ is a choice of square root line bundle
$K_{X,s}^{1/2}$ for $K_{X,s}$ on $X^\red$. That is, an orientation
is a line bundle $L$ on $X^\red$, together with an isomorphism
$L^{\ot^2}=L\ot L\cong K_{X,s}$. A d-critical locus with an
orientation will be called an {\it oriented d-critical locus}.
\label{da6def2}
\end{dfn}

\subsection{Proof of Theorem \ref{sm6thm5}}
\label{s4.2}

Let $(S,\om)$ be a complex symplectic manifold, and $L,M\subset S$ 
two complex Lagrangian submanifolds of $S$. 
Given the complex analytic space $X=L\cap M,$ we must construct a section $s\in H^0(S_X^0)$ 
such that $(X,s)$ is a complex analytic d-critical locus. 
We use notation from \S\ref{s3}, and in particular the notions of $L$-chart, $M$-chart, and $LM$-chart.

\medskip

We claim that there is a unique d-critical structure $s$ on $X,$ such that
\begin{enumerate}
\item every $L$-chart $(P,U,f,i)$ from a polarization $\pi_1 : S_1 \ra E_1$ transverse to $L,M$ is a critical chart on $(X,s);$
\item every $M$-chart $(Q,V,g,j)$ from a polarization $\pi_2 : S_2 \ra E_2$ transverse to $L,M$ is a critical chart on $(X,s).$
\end{enumerate}

where $L$-charts and $M$-charts are defined using transverse polarizations.
To show this we note that the $L$-chart $(P,U,f,i)$ determines a d-critical structure $s_P$ on $P,$ and similarly the $M$-chart $(Q,V,g,j)$ determines a d-critical structure $s_Q$ on $Q.$

\smallskip

Next, for given $L$-charts and $M$-charts, we use the $LM$-charts $(R,W,h,k)$ and Proposition \ref{local} in \S\ref{s3} to show that $s_P\vert_{P\cap Q} = s_Q\vert_{P\cap Q}.$ 

\smallskip

Then, we choose a locally finite cover of $L$-charts $(P_a,U_a,f_a,i_a)$ for $a \in A$ covering $X,$ from polarizations transverse to $L,M.$
We choose $M$-charts $(Q_b, U_b, f_b, i_b)$ for $b \in B$ covering $X,$ from polarizations transverse to $L,M$ and all polarizations used to define the $(P_a,U_a,f_a,i_a).$
Then we get: 
$s_{P_a}\vert_{P_a\cap Q_b} = s_{Q_b}\vert_{P_a\cap Q_b}$
for all $a,b.$
Hence $
s_{P_a}\vert_{P_a\cap P_{a'}\cap Q_b} = s_{P_{a'}}\vert_{P_a\cap P_{a'}\cap Q_b}$
for all $a,a'\in A,$ $b\in B.$
As the $Q_b$ cover $X,$ we have $
s_{P_a}\vert_{P_a\cap P_{a'}} = s_{P_{a'}}\vert_{P_a\cap P_{a'}},$
for all $a,a' \in A.$

\smallskip

So there exists a unique section $s$ with $s\vert_{P_a} = s_{P_a},$ for all $a \in A,$ as $S_X^0$ is a sheaf.
Finally, following the same technique of \S\ref{s3.3}, the construction is independence of choices.

\medskip

For the second part of the theorem, let $(P,U,f,i),$ be
 a critical
chart on $(X,s)$. Then Theorem 3.5(i) gives a
natural isomorphism
\e
\io_{P,U,f,i}:K_{X,s}\vert_{P^\red}\longra
i^*\bigl(K_U^{\ot^2}\bigr)\vert_{P^\red}.
\label{da6eq7}
\e
Using \eq{isocanbund.1}, note that $K_U^2 \cong K_L\ot K_M,$ as the polarization $\pi$ identifies both $L,M$ with $U$ locally, giving isomorphisms $K_U\vert_X \cong K_L\vert_X\cong K_M\vert_X.$
Now comparing with \eq{detL}, we get $K_{X,s}\vert_{P^\red}\cong\det(\bL_{X})\vert_{P^\red}$
for each $(P,U,f,i),$ critical
chart on $(X,s)$. 
Comparing two critical charts, one can show that the canonical isomorphisms constructed above 
from two such charts
 are equal on the
overlap. Therefore the isomorphisms 
glue to give a global canonical isomorphism
$K_{X,s}\cong\det(\bL_{X})\vert_{X^\red}$. This completes the
proof of Theorem \ref{sm6thm5}.

\medskip

Note that we did not use $LMLM$ charts and Proposition 2.3 in \S\ref{s3.2}.
That is because we are constructing a section $s$ of a sheaf, (effectively, a morphism in a category), rather than a (perverse) sheaf (an object in a category), so 
basically we only have to go up to double overlaps, not triple overlaps.

\section{Relation with other works and further research}
\label{s5}

In this section we briefly discuss related work in the literature, and outline some ideas for future investigation.

\medskip

\begin{center} \textbf{The work of Behrend and Fantechi \cite{BeFa}} \end{center}  

\smallskip

The main inspiration for the present work was a
result by Behrend and Fantechi \cite{BeFa} in 2006.
Their project aims to construct and study Gerstenhaber and
Batalin--Vilkovisky structures on Lagrangian intersections. 
They consider a pair $L,M,$ of complex Lagrangian submanifolds
in a complex symplectic manifold $(S,\om)$, and they show that one can equip the graded
algebra $\STor^{\cO_S}_{-i}(\cO_L,\cO_M)$ with a Gerstenhaber bracket, and
the graded sheaf $\SExt^i_{\cO_S}(\cO_L,\cO_M)$ with a Batalin--Vilkovisky type differential.
The approach is the same as our approach, and in fact we were inspired by that: it is 
based on the holomorphic version of the Darboux theorem, that is, 
any holomorphic symplectic manifold is locally isomorphic
to a cotangent bundle, thus reducing the case of a general Lagrangian intersection
to the special case where one of the two Lagrangian is identified with the zero section of the cotangent bundle of
the symplectic manifold, and the second one is the graph of a 
holomorphic function locally defined on the first Lagrangian. 

\smallskip

In particular, Behrend and Fantechi \cite[Th.s 4.3 \& 5.2]{BeFa} claim to construct
canonical $\C$-linear differentials
\begin{equation*}
\d:\cE xt^i_{\O_S}(\O_L,\O_M)\longra \cE xt^{i+1}_{\O_S}(\O_L,\O_M)
\end{equation*}
with $\d^2=0$, such that $\bigl(\cE xt^*_{\O_S}(\O_L,\O_M),\d\bigr)$
is a constructible complex, called the {\it virtual de Rham complex} of the Lagrangian intersection $X.$ Conjecturally, $(\cE^\bullet,\rm d)$ categorifies Lagrangian intersection numbers, in the sense that
the constructible function 
$$p\ra \sum_i (-1)^{i-\dim(S)} \dim_\C \mathbb{H}^i_{\{ p \}}(X, (\cE xt^\bu_{\O_S}(\O_L,\O_M), \d)),$$
of fiberwise Euler characteristic of $(\cE xt^\bu_{\O_S}(\O_L,\O_M), \d)$ is equal to the well known Behrend function $\nu_X$ in \cite{Behr}, and so $$\chi(X,\nu_X)=\sum_i (-1)^{i-\dim(S)}\dim_\C \mathbb{H}^i (X,(\cE xt^\bu_{\O_S}(\O_L,\O_M), \d)).$$

Their main theorem \cite[Th.~4.3]{BeFa} claims that the locally defined de Rham differentials coming from the picture given by the holomorphic Darboux theorem, do not depend on the way one writes $S$ as a cotangent bundle, or, in other words, that $\d$ is independent of the chosen polarization of $S.$ Thus, the locally defined $\d$ glue, and they obtain a globally defined canonical de Rham type differential.
Unfortunately, there is a mistake in the proof. To fix this one should instead work with $\cE
xt^*_{\O_S}(K_L^{1/2},K_M^{1/2})$ for square roots
$\smash{K_L^{1/2},K_M^{1/2}}$ as in \S\ref{s3}. Also
the relation between their virtual de Rham complex and vanishing cycles relies   
on a conjecture of Kapranov \cite[Rmk. 2.12(b)]{Kapr}, which later turned out to true just over the ring of 
Laurent series - see Sabbah \cite[Th.~1.1]{Sabb} 
(deformation--quantization setting, see discussion below).

\medskip\clearpage

\begin{center} \textbf{The work of Kashiwara and Schapira \cite{KaSc2}} \end{center}  

\smallskip

Kashiwara and Schapira \cite{KaSc3} develop a theory of {\it
deformation quantization modules}, or {\it DQ-modules}, on a complex
symplectic manifold $(S,\om)$, which roughly may be regarded as
symplectic versions of $\cD$-modules. {\it Holonomic\/} DQ-modules
${\cal D}^\bu$ are supported on (possibly singular) complex
Lagrangians $L$ in $S$. If $L$ is a smooth, closed, complex
Lagrangian in $S$ and $K_L^{1/2}$ a square root of $K_L$, D'Agnolo
and Schapira \cite{DASc} show that there exists a simple holonomic
DQ-module ${\cal D}^\bu$ supported on~$L$.

\smallskip

If ${\cal D}^\bu,\cE^\bu$ are simple holonomic DQ-modules on $S$
supported on smooth Lagrangians $L,M$, then Kashiwara and Schapira
\cite{KaSc2} show that $R{\scr H}om({\cal D}^\bu,\cE^\bu)[n]$ is a
perverse sheaf on $S$ over the field $\C((\hbar))$, supported on
$X=L\cap M$. Pierre Schapira explained to the authors of \cite{BBDJS} how to prove
that $R{\scr H}om({\cal D}^\bu,\cE^\bu)[n]\cong P_{L,M}^\bu$, when
$P_{L,M}^\bu$ is defined over the base ring~$A=\C((\hbar))$.

\medskip

\begin{center} \textbf{The work of Baranovsky and Ginzburg \cite{BaGi}} \end{center}  

\smallskip

Apart from the mistake in the proof, Behrend and Fantechi's work \cite{BeFa} gives a new important 
understanding of a rich structure on Lagrangian intersection, investigated also by Baranovsky and 
Ginzburg \cite{BaGi}, who obtained analogous results 
for any pair of smooth coisotropic submanifolds $L,M$ of arbitrary
smooth Poisson algebraic varieties $S$ considering first order deformations of the structure sheaf $\cO_S$ 
to a sheaf of non-commutative algebras and of the structure sheaves $\cO_L$ 
and $\cO_M$ to sheaves of right and left modules
over the deformed algebra. The construction is canonically defined and it is independent of the choices of deformations involved.

\smallskip

The proof of their main result, Theorem 4.3.1 in \cite{BaGi}, shows that sometimes the Gerstenhaber and
Batalin--Vilkovisky structures on Tor or Ext are
well-defined globally. In their construction, this is the case, for instance, whenever in the setting of the proof of
\cite[Thm 4.3.1]{BaGi}, some cocycles are defined globally.

\medskip

\begin{center} \textbf{The work of Kapustin and Rozansky \cite{KR2}} \end{center}  

\smallskip

In \cite{KR2}, Kapustin and Rozansky study boundary conditions and defects in a three-dimensional topological sigma-model with a complex symplectic target space, the Rozansky-Witten model.
It turns out that this model has a deep relation with the problem of deformation quantization of the derived category of coherent sheaves on a complex manifold, regarded as a symmetric monoidal category, and in particular with categorified algebraic geometry in the sense of \cite{BFN,TV}. Namely, in the case when the target space of the Rozansky-Witten model has the form of the cotangent bundle $T^*Y$, where $Y$ is a complex manifold, the $2$-category of boundary conditions is very similar to the $2$-category of derived categorical sheaves on $Y$.

\smallskip

More precisely, given a complex symplectic manifold $(S,\om)$, Kapustin and
Rozansky conjecture the existence of an interesting 2-category,
with objects complex Lagrangians $L$ with $K^{1/2}_L$, such that 
$\Hom(L,M)$ is a $\Z_2$-periodic triangulated category, and if $L\cap M$ is locally 
modeled on $\Crit(f : U \ra \C)$ for $f : U \ra \C$ is a holomorphic function on a manifold $U,$ then $\Hom(L, M)$ is locally modeled on the matrix factorization category $MF(U, f )$ as in \cite{Orlov}. 

\medskip

\begin{center} \textbf{Matrix factorization and second categorification} \end{center}  

\smallskip

It would be interesting to construct a sheaf of $\Z_2$-periodic triangulated categories on Lagrangian intersection,
which, in the language of categorification, would yield a second categorification of the intersection numbers, the first being given by the hypercohomology of the perverse sheaf constructed in the present work. 

\smallskip

Also, this construction should be compatible with the Gerstenhaber and Batalin--Vilkovisky structures 
in the sense of \cite[Conj. 1.3.1]{BaGi}.

\medskip\clearpage

\begin{center} \textbf{Fukaya category for derived Lagrangian and d-critical loci} \end{center}  

\smallskip

It would be interesting to extend Theorem \ref{sm6thm5} to a class of `derived
Lagrangians' in~$(S,\om)$.

\smallskip

Given a pair $L,M,$ of derived complex Lagrangian submanifolds
in the sense of \cite{PTVV}
in a complex symplectic manifold $(S,\om),$ with
$\dim_\C S=2n,$
Joyce conjectures that there should be some kind of
approximate comparison
\begin{equation*}
\bH^k(P_{L,M}^\bu)\approx HF^{k+n}(L,M),
\end{equation*}
where $HF^*(L,M)$ is the {\it Lagrangian Floer cohomology\/} of
Fukaya, Oh, Ohta and Ono \cite{FOOO}. Some of the authors of \cite{BBDJS}
are working on defining a `Fukaya category' of (derived)
complex Lagrangians in a complex symplectic manifold, using
$\bH^*(P_{L,M}^\bu)$ as morphisms.
See \cite{BBDJS} for a more detailed discussion.

\end{document}